\documentclass[11pt,a4paper]{amsart}

\usepackage[dvipsnames]{xcolor}
\usepackage[normalem]{ulem}
\usepackage{amsmath,amssymb,amsthm,graphics,amscd,mathrsfs}
\usepackage{longtable}
\usepackage{cite}
\usepackage[all,cmtip]{xy}
\usepackage{graphicx,wrapfig}
\usepackage{epsf}

\usepackage{hyperref}
\hypersetup{colorlinks=true,citecolor=Purple}

\usepackage{lscape}
\usepackage{tikz}
\usetikzlibrary{matrix,arrows.meta}

\usepackage{verbatim}

\renewcommand{\thesection}{\Roman{section}}

\usepackage{etoolbox}

\usepackage{titlesec}
\titleformat{\section}{\Large\bfseries\center\sc}
{Part\ \thesection\,---}{0.3ex}{} 
\titleformat{\subsection}{\large\sc\center}{\thesubsection.}{0.5ex}{}
\titleformat{\subsubsection}[runin]{\bf}{\thesubsubsection.}{0.5ex}{}[.]

  \renewenvironment{thebibliography}[1]{
    \begin{oldthebibliography}{#1}
      \setlength{\parskip}{0ex}
      \setlength{\itemsep}{0ex}
  }
  {
    \end{oldthebibliography}
  }
  
\setlongtables

\usepackage[capitalise]{cleveref}
\crefformat{equation}{#2(#1)#3}

\crefformat{enumi}{#2(\romannumeral #1\relax)#3}
\crefformat{enumii}{#2(\romannumeral\theenumi\relax)(\theenumii)#3}

\newtheorem{lemma}{Lemma}[subsection]
\newtheorem*{lemma*}{Lemma}
\newtheorem{theorem}[lemma]{Theorem}
\newtheorem*{theorem*}{Theorem}

\newtheorem{cor}[lemma]{Corollary}

\newtheorem{prop}[lemma]{Proposition}

\theoremstyle{definition}
\newtheorem{example}[lemma]{Example}

\newtheorem*{algorithm*}{Algorithm}
\theoremstyle{remark}
\newtheorem{remark}[lemma]{Remark}
\newtheorem{remarks}[lemma]{Remarks}
\newtheorem*{remarks*}{Remarks}

\theoremstyle{definition}
\newtheorem{defn}[lemma]{Definition}

\theoremstyle{plain}
\newtheorem{theorem1}{Theorem}
\crefformat{theorem1}{Theorem #2#1#3}

\makeatletter
\newcommand{\pushright}[1]{\ifmeasuring@#1\else\omit\hfill$\displaystyle#1$\fi\ignorespaces}
\newcommand{\pushleft}[1]{\ifmeasuring@#1\else\omit$\displaystyle#1$\hfill\fi\ignorespaces}
\makeatother

\DeclareMathOperator{\Rad}{Rad}

\DeclareMathOperator{\rk}{rk}
\DeclareMathOperator{\id}{id}

\DeclareMathOperator{\Spec}{Spec}

\DeclareMathOperator{\Char}{char}
\DeclareMathOperator{\Res}{Res}
\DeclareMathOperator{\Soc}{Soc}
\DeclareMathOperator{\Ind}{Ind}
\DeclareMathOperator{\Gal}{Gal}
\DeclareMathOperator{\Mat}{Mat}
\DeclareMathOperator{\Ann}{Ann}
\DeclareMathOperator{\Jac}{Jac}

\DeclareMathOperator{\Gr}{Gr}

\newcommand{\R}{\mathrm{R}}
\newcommand{\RR}{\mathscr{R}}
\newcommand{\G}{\mathscr{G}}
\newcommand{\GG}{\mathcal{G}}
\newcommand{\TT}{\mathcal{T}}

\newcommand{\Hom}{\mathrm{Hom}}

\newcommand{\D}{\mathscr{D}}

\newcommand{\SL}{\mathrm{SL}}
\newcommand{\SU}{\mathrm{SU}}
\newcommand{\GL}{\mathrm{GL}} 
\newcommand{\kp}{k_{\mathbf{p}}}

\newcommand{\End}{\mathrm{End}}

\newcommand{\Aut}{\mathrm{Aut}}

\newcommand{\NN}{\mathbb{N}}
\newcommand{\Z}{\mathbb{Z}}

\newcommand{\F}{\mathbb{F}}
\newcommand{\m}{\mathfrak{m}}

\newcommand{\Gm}{\mathbb{G}_m}

\newcommand{\Q}{\mathbb{Q}}

\newcommand{\Ga}{\mathbb{G}_a}

\newcommand{\KK}{\mathcal{K}}
\newcommand{\Cs}{\mathsf{C}}
\newcommand{\ks}{\mathsf{k}}
\newcommand{\Ts}{\mathsf{T}}
\newcommand{\CCC}{\mathcal{C}}

\newcommand{\sep}{\mathrm{sep}}

\newcommand{\C}{\mathscr{C}}
\newcommand{\M}{\mathscr{M}}
\renewcommand{\D}{\mathscr{D}}
\newcommand{\U}{\mathscr{U}}

\newcommand{\bark}{{\overline{k}}}
\newcommand{\barG}{{\overline{G}}}

\title{Geometric rigidity of simple modules for algebraic groups}
\subjclass[2020]{Primary 20G05, 20G15; Secondary 16K20, 13C05}

\author[Bate]{Michael Bate}
\address
{Department of Mathematics,
Ian Wand Building,
University of York,
York YO10 5GH,
UK}
\email{michael.bate@york.ac.uk}
\author[Stewart]{David I. Stewart}
    \address{Department of Mathematics, The University of Manchester, Manchester, UK}
    \email{david.i.stewart@manchester.ac.uk}

\begin{document}

\begin{abstract}
Let $k$ be a field, let $G$ be an affine algebraic $k$-group and $V$ a finite-dimensional $G$-module. 
We say $V$ is \emph{rigid} if the socle series and radical series coincide for the action of $G$ on each indecomposable summand of $V$; say $V$ is \emph{geometrically rigid} (resp.~\emph{absolutely rigid}) if $V$ is rigid after base change of $G$ and $V$ to $\bark$ (resp.~any field extension of $k$). 
We show that all simple $G$-modules are geometrically rigid, though not in general absolutely rigid.
More precisely, we show that if $V$ is a simple $G$-module, then there is a finite purely inseparable extension $k_V/k$ naturally attached to $V$ such that
$V_{k_V}$ is absolutely rigid as a $G_{k_V}$-module. 
The proof turns on an investigation of algebras of the form $K\otimes_k E$ where $K$ and $E$ are field extensions of $k$; 
we give an example of such an algebra which is not rigid as a module over itself. 
We establish the existence of the purely inseparable field extension $k_V/k$ through an analogous version for artinian algebras.

In the second half of the paper we apply recent results on the structure and representation theory of pseudo-reductive groups to give a concrete description of $k_V$ when $G$ is smooth and connected. Namely, we combine the main structure theorem of the Conrad--Prasad classification of pseudo-reductive $G$ together with our previous high weight theory. For a simple $G$-module $V$, we calculate the minimal field of definition of the geometric Jacobson radical of $\End_G(V)$ in terms of the high weight of $V$ and the Conrad--Prasad classification data; this gives a concrete construction of the field $k_V$ as a subextension of the minimal field of definition of the geometric unipotent radical of $G$. 

 We also observe that the Conrad--Prasad classification can be used to hone the dimension formula for $V$ we had previously established; we also use it to give a description of $\End_G(V)$ which includes a dimension formula.
\end{abstract}

\maketitle

\section*{Introduction}

Let $k$ be a field and $G$ an affine algebraic $k$-group.
The recent classification by highest weight of the (rational) simple $G$-modules for smooth connected $G$ in \cite{BS22} has opened the possibility of answering general
questions about the representation theory of algebraic groups, which hitherto might have seemed inaccessible; this paper is presented in that spirit. 
Given a simple $G$-module $V$ we provide rather detailed information about the behaviour of $V$ under field extensions---we describe the structure of $V_{E}$ as a $G_{E}$-module for suitable field extensions $E/k$. 
Principally, we prove that whilst $V$ is far from being absolutely simple in general, or even absolutely semisimple, it is at least \emph{geometrically rigid}. For any finite-dimensional $G$-module $V$, we say 
that $V$ is \emph{rigid} if the socle series and radical series coincide for the action of $G$ on each of its indecomposable summands; we say $V$ is \emph{geometrically rigid} (resp.~\emph{absolutely rigid}) if $V$ is rigid after base change to $\bark$ (resp.~any field extension of $k$). Explicit definitions of the above are to be found in \cref{sec:soc}.

Our main result is
\begin{theorem1}\label{thm:mainone}
Let $G$ be an affine algebraic $k$-group and $V$ a simple $G$-module, and let $k_V/k$ denote the minimal field of definition of the Jacobson radical of $\End_G(V)_{\bark}$.
Then $k_V/k$ is purely inseparable and the $G_{k_V}$-module $V_{k_V}$ is absolutely rigid.
In particular, $V$ is geometrically rigid.
\end{theorem1}

\begin{remarks*}
  \begin{enumerate}\item The notion of rigidity has mostly been investigated in the context of indecomposable modules and there is a choice as to how best to generalise it. The obvious alternative would be to say $V$ is \emph{rigid\,$'$} if its socle and radical series coincide; however, this would rule out the direct sums of rigid$'$ modules being rigid$'$ unless their socle series had the same length. Defined as we have, though, the category of rigid $G$-modules is additive. (Our definition also has the advantage of paving the way for the generalisation to $k$ being an algebra, where we would assert that the fibres $V_K$ should be rigid $G_K$-modules.) 

\item One of the earliest results about rigidity is due to Jennings: over a field, the group algebra of a finite $p$-group is rigid---\!\cite[p93]{Ben98}. 

\item Obviously any simple module $V$ is rigid, but even when $G$ is connected and smooth, $V$ is not in general absolutely semisimple or absolutely indecomposable, so a statement about how rigidity behaves under base change is not immediate---see Remark \ref{rem:notabs} below. 
It turns out that one can find examples $(G,V,L/E/k)$, where $G$ is a smooth connected $k$-group, $V$ is a simple $G$-module, and $L/E/k$ is a tower of finite extensions, such that $V_E$ is \emph{not} rigid but $V_L$ is absolutely rigid. 
This boils down to exhibiting a tensor product of finite purely inseparable field extensions $K$ and $E$ of $k$ such that  the algebra $A:=K\otimes_k E$ is not rigid as a module over itself---see Example \ref{ex:finally}. 

\item A natural source of indecomposable non-simple modules for a split reductive group $G$ over a field $k$ are Weyl modules $V(\lambda)$ and tilting modules $T(\lambda)$, where we refer the reader to \cite[II.2.13,\ II.8.3,\ E.3]{Jan03} for definitions. 
It is natural to ask when these modules are rigid. This question is given a thorough treatment in \cite{And11}, and one finds that for characteristic sufficiently large, both Weyl modules and tilting modules are indeed rigid. Non-rigid examples of both are provided in \emph{op.~cit.} as are more examples, again for $\SL_3$, in \cite{BDM} when $p=3$. See also \cite{Haz17} for a novel approach to this problem.
\end{enumerate}
\end{remarks*}

We divide into two parts. Part \ref{parti} is dedicated to the proof of \cref{thm:mainone}; here we show the result via a rather general argument involving the behaviour of finite-dimensional simple $k$-algebras under base change.
Given such a $k$-algebra $A$, we show that the Jacobson radical $\Jac(A_\bark)$ of the base change of $A$ to $\bark$ has a descent to a minimal field of definition $k'$ where $k'/k$ is purely inseparable; moreover, $A_{k'}$ is absolutely rigid; see \Cref{thm:Drigid}.
This general result implies \cref{thm:mainone} via an argument which shows that in the situation of the theorem it is enough to consider the division ring $\End_G(V)$. (In fact, a further reduction allows us to replace $\End_G(V)$ with its centre, which is some finite field extension of $k$.)
During this part of the paper, we also construct an example of two finite purely inseparable extensions $K$ and $E$ such that the regular module of $K\otimes_k E$ is not rigid and deduce a large class of $G$-modules which are not absolutely rigid.

Part \ref{partii} sharpens the conclusion of \cref{thm:mainone} using the high weight theory of \cite{BS22} for pseudo-reductive groups together with the Conrad--Prasad structure theorem describing their classification, 
\cite[Thm.~9.2.1]{CPClass}. We consider endomorphism rings of simple modules in \cref{sec:endsimple1,sec:endsimple2}, whose deliberations afford a concrete construction of $k_V$.
Essentially this reduces to the case where $G$ is a pseudo-split pseudo-reductive group with no non-trivial normal unipotent $k$-subgroup scheme, hence locally of minimal type, hence described by the Conrad--Prasad structure theorem. (Recall that $G$ is \emph{pseudo-split} if it has a split maximal torus.) Then by \cite{BS22}, a simple module $V$ is isomorphic to $L_G(\lambda)$ and the field $k_V$ we use actually coincides with $\End_G(V)$, which also identifies with the high-weight space $L_G(\lambda)_\lambda$; it can be precisely described as a compositum of purely inseparable field extensions using arithmetic information about $\lambda$ together with the Conrad--Prasad data defining $G$. Mostly, the root system of $G$ has no bearing on $k_V$, while evidently it does on $L_G(\lambda)$. 
In case $G$ is an arbitrary smooth connected affine algebraic $k$-group, we elucidate the structure of the division algebra $D:=\End_G(V)$. We show that $D$ has a unique \emph{$p$-splitting field}: there is a unique minimal extension $E/k$ such that $D\otimes_k E$ is a product of matrix algebras over purely inseparable extensions of $E$. 
Thus engaged, we interpret our previous dimension formula in terms of the same data, and give a formula for the dimension of $D$: \cref{cor:dimforsimples}.

A useful auxiliary result locates the simple modules for pseudo-split pseudo-reductive groups as submodules of simple modules for Weil restrictions of reductive groups. More specifically, for a pseudo-reductive group $G$, there is a homomorphism $i_G:G\to \R_{k'/k}(G')$, where $k'$ is the minimal field of definition for the geometric unipotent radical of $G$; the group $G'$ is the corresponding reductive quotient of $G_{k'}$; and $\R_{k'/k}$ denotes the Weil restriction functor. We show that when $G'$ is pseudo-split, the simple modules for $\R_{k'/k}(G')$ are semisimple and isotypic upon restriction to the image of $G$; see \cref{prop:splitend}.

\section[{Proof of \texorpdfstring{\cref{thm:mainone}}{Theorem 1}}]{Proof of \texorpdfstring{\cref{thm:mainone}}{Theorem 1}}\label{parti}
\subsection{Preliminaries}\label{sec:prelims}

Our main references for the theory of algebraic groups are \cite{CGP15}, \cite{Milne17}, and \cite{Jan03}, with the last also our standard reference for
the representation theory of algebraic groups. 
In most of the paper, $k$ denotes a field, but below we do need to consider the base change of $k$-groups and modules to more general $k$-algebras, so up until \cref{cor:DEZE} we also let $k$ denote a general commutative unital ring. 
For such a $k$, we view an affine $k$-group $G$ as a functor $\underline{k\textrm{-Alg}}\to \underline{\mathrm{Grp}}$
which is represented by a $k$-algebra $k[G]$; in other words $G(?)\cong \Hom_{k\textrm{-Alg}}(k[G],?)$. 
If $k[G]$ 
is isomorphic to $k[T_1,\dots,T_n]/I$ for $I$ finitely generated, then we say $G$ is finitely presented or \emph{algebraic} 

\begin{center}\framebox{In what follows $G$ will always denote an affine algebraic $k$-group scheme.}\end{center}

In particular, suppose that $k$ is a field, $k_s$ its separable closure, and $\bark$ its algebraic closure.
If $G$ is smooth, then $G$ is geometrically reduced, and it follows from \cite[Cor. 1.17]{Milne17}
that $G_{k_s}(k_s)$ is dense in $G_{k_s}$.

\subsubsection{Modules for algebraic groups}\label{sec:Gmods} 
Recall that $k$ is a commutative unital ring.
Let $M$ be a $k$-module (possibly not finitely generated). 
Then we may define a group functor $M_a:\underline{k\textrm{-Alg}}\to \underline{\mathrm{Grp}}$ so that 
$M_a(A)=M\otimes_k A$ inherits a group structure from the additive group on $A$. 
Note that, even when $k$ is a field, $M_a$ is only an algebraic group when $M$ is finite-dimensional. 
Recall that a left action of $G$ on a $k$-functor $X$ is a morphism (i.e. a natural transformation) 
$\phi:G\times X\to X$ such that $\phi(A):G(A)\times X(A)\to X(A)$ is a left action of the group $G(A)$ on $X(A)$ for each $k$-algebra $A$. 
In case $G$ acts on $M_a$ such that the action of $G(A)$ on $M_a(A)$ is $A$-linear for each $k$-algebra $A$, we say $M$ is a \emph{representation for $G$}, or more frequently in this paper, a \emph{$G$-module}.
These definitions allow us to work with arbitrary $k$-modules, 
although in the case $M\cong k^n$ is a $G$-module 
then it corresponds to a homomorphism $G\to \GL(M)$ of algebraic groups.  A $G$-module $M$ is equivalently a comodule for the Hopf algebra $k[G]$ and we denote the comodule map $\Delta_M:M\to M\otimes k[G]$. 
For example when $M$ is a finitely generated free $k$-algebra with basis $\{e_1,\dots, e_n\}$, then the natural representation of $\GL(M)$ on $M$ corresponds to the comodule map determined on $e_i$ by $\Delta_M(e_i)=\sum_{1\leq j\leq n}(e_j\otimes T_{ji})$ where $k[\GL(M)]=k\left[T_{ij},\frac{1}{\det} \mid 1\leq i,j\leq n\right]$. A $G$-submodule is a $k$-submodule $N\subseteq M$ 
such that $N_a(A)$ is $G(A)$-stable for all $k$-algebras $A$. If $G$ is flat---which is to say that $k[G]$ is a flat $k$-module---then  $N$ is a $G$-submodule if and only if $\Delta_M(N)\subseteq N\otimes k[G]$. If $k\to E$ is a homomorphism of rings and $M$ is a $G$-module then $M_{E}:=M\otimes_k E$ acquires an action of the base change $G_{E}$ of $G$ making it into a 
$G_{E}$-module. 
For more on these definitions, see \cite[\S I.2]{Jan03}.

Morphisms between $G$-modules are $G$-equivariant $k$-linear maps. 
If $M$ and $N$ are $G$-modules, the full collection of such morphisms is denoted $\Hom_{G}(M,N)$. If $M=N$, we write $\End_G(M)$ instead. If $G$ is flat then the category of $G$-modules is abelian; i.e. kernels and cokernels are submodules \cite[I.2.9]{Jan03}. Therefore, we have the following, with the usual proof.
\begin{lemma}[Schur] Suppose that $G$ is flat and let $M$ be a simple $G$-module. Then $\End_G(M)$ is a division ring.\label{lem:schur}\end{lemma}
Recall a $G$-module $M$ is \emph{locally finite} if any element $m\in M$ is contained in a $G$-submodule of $M$ which is finitely generated as a $k$-module. 
It has been noticed by Wilberd van der Kallen \cite[\S1.7]{vdk21} that the proof in \cite[I.2.13]{Jan03} that $G$-modules are locally finite for arbitrary $k$ and flat $G$ is incomplete. It relies on the assumption that an arbitrary intersection of $G$-submodules is again a submodule; however, there is a counterexample to be found at \cite[Exp.~VI, Remarque 11.10.1]{SGA3}. 
We are grateful to Ofer Gabber for providing the following example which shows that local finiteness can indeed fail for arbitrary flat $G$.

\begin{example}
Let $k$ be a rank one valuation ring and assume $k$ is not a DVR. Any such $k$ is not noetherian and it arises as the subring $\{x\in \mathrm{Frac}(k)\mid v(r)\geq 0\}$ for some valuation $v:\mathrm{Frac}(k)\setminus\{0\}\to \mathbb{R}$ with dense image.
(For a concrete example,  one could take $k$ to be the valuation ring in the field $F(X^{1/2^n}\mid n\in \mathbb{N})$ with valuation induced by the degree function, where $F$ is any field.) 
We have that $k$ is local with unique maximal ideal $\m=\{x\in k\mid v(x)>0\}$, which therefore has $\m^2=\m$. 

Let $k[\Gm]=k[T,T^{-1}]$ be the coordinate ring of $\Gm$. This is a Hopf algebra with $\Delta(T)=T\otimes T$, $S(T)=T^{-1}$, $\epsilon(T)=1$. Consider the subring $R\subset k[T,T^{-1}]$ of elements for which the coefficients of powers $T^n$ for nonzero $n$ lie in $\m$. 
This is easily seen to be a sub-Hopf algebra of $k[\Gm]$, so defines an affine $k$-group $G$ and an associated dominant morphism $\Gm\to G$. Indeed, $G$ is flat---which is to say that $R$ is a flat $k$-module; this follows since $k$ is a valuation ring and both $\m$ and $k$ are torsion-free $k$-modules \cite[\href{https://stacks.math.columbia.edu/tag/0549}{Tag 0549}]{stacks-project}.
Let $V:=k\cdot e$ be the standard representation of $\Gm$, whose comodule map is $\Delta_V:V\to V\otimes k[T,T^{-1}];\ e\mapsto e\otimes T$. 
Then $M:=\m\cdot e$ is a $G$-submodule using the characterisation of \cite[I.2.9(1)]{Jan03}: for we have $$\Delta_V(M)=\m e\otimes T=\m^2e\otimes T=\m e\otimes \m T \subseteq M\otimes R.$$ 
By the same token, if $N$ is a finitely generated $G$-submodule of $M$, then we can show that $N=0$.
For suppose we have such an $N$ which is nonzero. 
Then $N$ is of the form $\mathfrak{n}\cdot e$ for some finitely generated ideal $\mathfrak{n}=(n_1,\dots,n_r)$ of $\m$. 
Further, in fact $\mathfrak{n}=(n)$ is principal, where $0\neq n$ is any element of $\mathfrak{n}$ with $v(n)$ minimal. But since $\Delta_V(n\cdot e)=ne\otimes T\in N\otimes R$, we have $ne\otimes T=n'e\otimes mT$ for $m\in\m$ with necessarily $v(m)>0$, which implies that $v(n')<v(n)$, contradicting the minimality of $v(n)$. Hence the only $G$-submodules of $M$ are not finitely generated. Thus $M$ is not locally finite. 
\end{example}

\begin{remark}
  We are grateful to the referee for the observation that the above is also an example of a non-finitely generated flat algebra whose fibres have different dimensions.
\end{remark}

Note also that local finiteness is used to infer finite generation over $k$ of a simple $G$-module, 
which means that the proof of \cite[I.10.15]{Jan03} is also incomplete.
However, if $G$ is flat and $k$ is noetherian then the proof in \emph{loc.~cit.} goes through; that is, under these hypotheses all $G$-modules are locally finite and, in particular, simple modules are finitely generated over $k$.

For a flat $k$-algebra $E$ (e.g., an extension of fields $E/k$) and $G$-modules $M$ and $N$ such that $M$ is finitely generated and projective as a $k$-module, we have \cite[I.2.10(7)]{Jan03}
\begin{equation}\label{eq:tensorhom}
\Hom_G(M,N)\otimes_k E \cong \Hom_{G_{E}}(M_{E},N_{E}). 
\end{equation}

\subsubsection{Socle, Radical, Rigidity}\label{sec:soc}
Let $R$ be a ring (not necessarily commutative) and $M$ an $R$-module of finite length; i.e.~a module with finitely many composition factors. Recall that the \emph{socle} $\Soc_R(M)$ of $M$ is defined to be the sum of its simple submodules. The \emph{socle series} (or \emph{Loewy series}) of $M$ is defined recursively by setting $\Soc^1M = \Soc_RM$
and letting $\Soc^iM$ be the submodule such that $\Soc^iM/\Soc^{i-1}M = \Soc_R(M/\Soc^{i-1}M)$. Dually, the \emph{radical} $\Rad_RM$ is the intersection of the maximal proper submodules of $M$. The radical series is defined recursively by letting $\Rad^1M = \Rad_RM$ and then $\Rad^iM:=\Rad_R(\Rad^{i-1}M)$. By convention, we set $\Rad^0(M) = M$. Taking $R$ as a module over itself, we have obviously $\Jac(R)=\Rad(R)$.

We say the socle series (resp.~radical series) has length $\ell:=\ell(M)$ if $\ell\in\mathbb N$ is minimal such that $\Soc^\ell(M)=M$ (resp.~$\Rad^\ell(M)=0$). 
They must have a common length $\ell$ called the \emph{Loewy length} and we always have an inclusion $\Rad^{\ell-i}(M) \subseteq \Soc^i(M)$ \cite[\S9.4]{ANT}.
 
If $M$ is indecomposable with Loewy length $\ell$---for example when it is the regular module for a local ring---then we say $M$ is \emph{rigid} if the radical and socle series coincide, that is $$
\Soc^iM = \Rad^{\ell-i}M \textrm{ for each } 0\leq i \leq \ell.$$ We say an $R$-module $M$ is \emph{rigid} if its indecomposable summands are so. 
If $R$ is a $k$-algebra for a field $k$, then $M$ is called \emph{geometrically rigid} if $M_{\bark}$ is a rigid $R_{\bark}$-module. It is called \emph{absolutely rigid} if $M_{E}$ is a rigid $R_E$ module for any field extension $E$ of $k$.

Following \cite[I.2.14, II.D.1]{Jan03}, we can replace $R$ with $G$ in all the above, where $G$ is an affine algebraic $k$-group and $k$ is a field. Then we get obvious notions of radical and socle series of finite-dimensional $G$-modules, and the idea of when one is rigid.

The following collects some basic observations about rigidity, whose proofs are obvious:

\begin{lemma}\label{sumrigid}
Suppose that $M$ is a finite-length $R$-module or $G$-module, and $M = U_1\oplus \cdots \oplus U_r$ is a decomposition of $M$ as a direct sum of submodules.
\begin{itemize}
\item[(i)] $\Soc^j(M) = \bigoplus_{i=1}^r \Soc^j(U_i) \textrm{  and } \Rad^j(M) = \bigoplus_{i=1}^r \Rad^j(U_i).$
\item[(ii)] The socle and radical series for $M$ coincide if and only if they coincide for each summand \emph{and} the summands all have a common Loewy length.
\item[(iii)] If the socle and radical series for $M$ coincide, then $M$ is rigid.
\end{itemize}
\end{lemma}

\subsubsection{On artinian algebras}\label{sec:artin} We need a little non-commutative algebra and our sources are \cite{Lam1} and \cite{Lam2}.
Let $k$ be a commutative artinian ring and $A$ a (possibly non-commutative) $k$-algebra which is finitely generated as a $k$-module (and hence artinian).
Let $\Jac(A)$ denote the Jacobson radical of $A$: the intersection of all maximal left ideals. It can be shown that $\Jac(A)$ is the annihilator of all simple left $R$-modules, from which it follows it is  a two-sided ideal of $A$---\hspace{1sp}\cite[\S4]{Lam1}.
Since $A$ is artinian, $\m_A:=\Jac(A)$ is the maximal nilpotent ideal of $A$ and the quotient $A/\m_A$ is the maximal semisimple quotient of $A$.
Further, an $A$-module $V$ is semisimple if and only if it is annihilated by $\m_A$---\hspace{1sp}\cite[Ex.~4.18]{Lam1}. 
So we see in this case that the terms of the socle series in a module $V$ are the annihilators of the powers of $\m_A$; \emph{viz.~}
\[\Soc^i(V)=\{v\in V\mid (\m_A)^i\cdot v=0\};\quad \Soc^i(A)=\Ann_A((\m_A)^i) ,\]
where we consider $A$ as a left regular $A$-module. If in addition $A$ is indecomposable \cite[\S22]{Lam1} and $n$ is the Loewy length of $A$, then we have
\begin{equation}\text{$A$ is rigid if and only if $(\m_A)^{n-i} = \Ann_A((\m_A)^i)$ for each $i$.}\label{rigidann}\end{equation}
If $A$ is not indecomposable, then $A$ is rigid if and only if the equality in \eqref{rigidann} holds for each of the blocks (indecomposable summands) of $A$.
Note that, as in \Cref{sumrigid}, for a general artinian algebra $A$, if the equality in \eqref{rigidann} holds, then $A$ is rigid in such a situation;
the converse is not true in general (e.g., see Example \ref{rem:notbottom}).

In the commutative case, a property related to rigidity is that of being Gorenstein.
Recall that a zero-dimensional commutative local noetherian ring $A$ is Gorenstein if one of the following equivalent conditions holds \cite[Prop. 21.5]{Eisenbud}: $A$ has a simple socle as a left $A$-module; $A$ is self-injective. 
Or if $A$ is a $k$-algebra then equivalently $A$ is self-dual over $A$, that is, $A\cong \Hom_k(A,k)$.

\begin{example}\label{rem:rigidegs}
We describe a non-rigid commutative artinian Gorenstein algebra. Let $k$ be any field and consider $A:=k[x,y]/(xy,x^3-y^2)$ as a module over itself. Then there is a unique maximal ideal $\m=\Jac(A)=(x,y)$.
So we see that $\Rad^i(A)$ coincides with  $\m^i$, and the nonzero ones are $\m$, $\m^2=(x^2)$, and $\m^3=(x^3)$. 
Meanwhile, $\Soc^i(A)$ coincides with the annihilators $\Ann_A(\m^i)$; we have $\Soc(A)=(x^3)=\m^3$ and $\Soc^3(A)=(x,y) = \m$, but $\Soc^2(A)=(x^2,y)\neq \m^2$. 
So even though $A$ is Gorenstein (since, for example, it has a simple socle), $A$ is not rigid. 

Pictorially, the relationship of the socle and radical series is offered in  \cref{nonrigpic}, where the lines show successive powers of $x$ and $y$ and $\Soc^i(A)$ and $\Rad^i(A)$ are found by taking the ideal generated by all elements occuring at that level or below. 

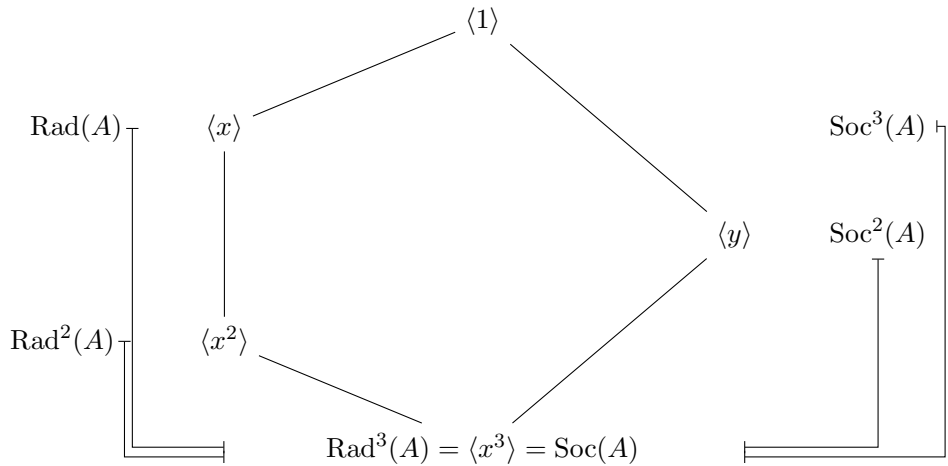
\begin{figure}\begin{center}\begin{tikzpicture}[scale=0.95,every node/.style={scale=0.95}]
  \matrix (m)
    [
      matrix of math nodes,
      nodes in empty cells,
      row sep    = 2em,
      column sep = 2em
    ]
    {
       & &  (  1 )  & \\
      \hspace{1em}\Rad(A) &  (  x )  & &  & \Soc^3(A)&[-2em]\\
       & & &  (  y  )  & \Soc^2(A)\\
      \Rad^2(A)&  (  x^2 )  & &  \\
        &\node(c){};& \Rad^3(A)= (  x^3 ) =\Soc(A) & & & \node(a) { }; \\
    };
    \draw (m-2-2) -- (m-1-3);
    \draw (m-2-2) -- (m-4-2);
    \draw (m-4-2) -- (m-5-3);
    \draw (m-1-3) -- (m-3-4);
    \draw (m-3-4) -- (m-5-3);
    \draw[{Bar}-{Bar}] (m-2-5) -| (a.center) |- (m-5-4);
    \draw[{Bar}-{Bar}] (m-3-5.south) |- (m-5-4.north east);
    \draw[{Bar}-{Bar}] (m-4-1.east) |- (c.center);
    \draw[{Bar}-{Bar}] (m-2-1.east) |- (c.north);
  
\end{tikzpicture}\end{center}\caption{A non-rigid algebra}\label{nonrigpic}\end{figure}
\end{example}

It was shown in \cite[\S70]{Mac} that one may characterise rigidity of a commutative Gorenstein  algebra of finite dimension over a field $k$ using its Hilbert function, so we explain this now. Take any filtration $\{M_i:=\mathscr{F}_i(M)\}$, i.e.~$M=M_0\supset M_1 \supset M_2 \supset \dots$, where the $M_i$ are all submodules of $M$. Then we may define the associated Hilbert function 
\[H_{\mathscr{F}}(M)(x):=H_{\mathscr{F},0}+H_{\mathscr{F},1}x+H_{\mathscr{F},1}x^2+\dots\] where $H_{\mathscr{F},i}=\dim_k(M_i/M_{i+1})$. 

Let $M=A$ and $\mathscr{F}$ the radical filtration $A_i:=(\m_A)^i$. 
Then Macaulay showed that the socle and radical series for $A$ coincide if and only if the coefficients of $H_{\mathscr{F}}$ are symmetric about the middle. 
Rather similarly, the socle and radical series for $A$ coincide if and only if there is an isomorphism $A\cong\Gr_\m(A)$, where $\Gr_\m(A)$ denotes the associated graded algebra of $A$ arising from the radical filtration. There is a detailed study of the Hilbert functions that can arise in \cite{Iar94}.

\subsubsection{Base change}
Let $k$ be a field and $G$ an affine algebraic $k$-group; let $A$ be a finite-dimensional $k$-algebra. We record two results about the behaviour of $G$-modules and $A$-modules under base change. 
Let $V$ be a finite-dimensional $A$-module (so $V$ is finite-dimensional with respect to the $k$-vector space structure inherited from $A$). 
Then obviously for any field extension $E/k$, we have $V_E:=V\otimes_kE$ is an $A_E:=A\otimes_k E$-module in a natural way. If $W$ is another $A$-module then, analogously with \cref{eq:tensorhom}, the finite-dimensionality of $V$ and the flatness of $E/k$ implies \[\Hom_A(V,W)_E\cong\Hom_{A_E}(V_E,W_E).\] Now we may state 

\begin{lemma}\label{lem:semisimpledown}
Suppose that $E/k$ is an extension of fields  and $V$ is a finite-dimensional $G$-module (resp. $A$-module).
If $V_{E}$ is semisimple as a $G_{E}$-module (resp.~ $A_E$-module), then $V$ is semisimple as a $G$-module (resp.~$A$-module).
\end{lemma}

\begin{proof}
We consider the case of $G$-modules; the proof for $A$-modules is identical. We prove the contrapositive.
Suppose that $V$ is not semisimple.
Then there is a non-split short exact sequence $0\to M \to V \to U \to 0$ with $U$ simple.
Since $U$ is simple, every element of $\Hom_G(U,V)$ must actually have image in the submodule $M$; that is
$\Hom_G(U,V) = \Hom_G(U,M)$.
Tensoring with $E$, we get from \eqref{eq:tensorhom} that 
$\Hom_{G_{E}}(U_{E},V_{E}) = \Hom_{G_E}(U_E,M_E)$.
In other words, the short exact sequence $0\to M_{E} \to V_{E} \to U_{E} \to 0$ also has no splitting, so $V_{E}$ is not semisimple.
\end{proof}

The first part of the next result provides a partial converse to \cref{lem:semisimpledown}; it is well-known, but we include a proof for completeness.

\begin{lemma}\label{lem:sepext}
Suppose that $E/k$ is a separable algebraic field extension, and $V$ is a finite-dimensional $G$-module or $A$-module.
\begin{itemize}
\item[(i)] $V$ is semisimple if and only if $V_{E}$ is semisimple.
\item[(ii)] $\Soc^i(V_{E})=(\Soc^i(V))_{E}$ and $\Rad^i(V_{E})=(\Rad^i(V))_{E}$.
\end{itemize}
\end{lemma}

\begin{proof} Again we only give the proof for $G$-modules, with the proof for $A$-modules being completely analogous.
Suppose first that $E$ is separably closed and $V$ is simple.
Then $\Soc_{G_E}(V_E)$ is a non-trivial $G_E$-submodule of $V_E$ which is stable under the Galois group of $E/k$, 
and so has a $k$-form.
That is, there is a $G$-submodule $U$ of $V$ with $U_E = \Soc(V_E)$.
Since $V$ is simple and $\Soc_{G_E}(V_E)$ is non-trivial, we have $U = V$.
Using \cref{lem:semisimpledown}, we can now deduce for all separable extensions $E/k$ and all finite-dimensional $G$-modules $V$ that we have $\Soc_{G_E}(V_E)= \Soc_G(V)_E$.
This is enough to deduce (i) and the first part of (ii).
The second part of (ii) follows by repeating the argument with the radical in place of the socle.
\end{proof}

\subsubsection{Minimal fields of definition for Jacobson radicals}\label{sec:minfield}
\begin{defn}Let $A$ be a finite-dimensional $k$-algebra, $K/k$ be a field extension and $M$ an $A_K$-module. An intermediate field $K/E/k$ is a \emph{field of definition} for $M$ if there is an $A_E$-module $N$ such that $N_K$ and $M$ are isomorphic $A_K$-modules. If $E$ is a field of definition for $M$ admitting no proper subfield of definition, then $E$ is a \emph{minimal field of definition} for $M$.

  Similarly, if $M$ is an $A$-module and $M'$ is an $A_K$-submodule of $M_K$, then $E$ is a field of definition for $M'$ if there is a submodule $N$ of $M_E$ such that $N_K=M'$, and it is minimal if $E$ admits no proper subfield of definition.
\end{defn}

\begin{remark}In general there is no guarantee of a minimal field of definition: in \cite[Sec.~6]{BR19} one can find an example of a two-dimensional module for the quaternion algebra $A:=\Q\{x,y\}/(x^2=y^2=-1,\ xy=-yx)$ that is defined over the field $K=\Q(a,b)/(a^2+b^2+1)$, and whenever it is defined over some subfield $E$ of $K$, it is also defined over some proper subfield of $E$. \end{remark}

In the special case $A=k$, so that $M$ is a $K$-vector space, one sees that $M\cong N_K$ if $N$ is a $k$-vector space with a basis of the same cardinality as that of $M$ so that its minimal field of definition exists and is $k$. When $V$ is any $k$-vector space and $W$ a $K$-subspace for some field extension $K/k$, it is explained in \cite[Rk.~1.1.7]{CGP15} how to construct the unique minimal field of definition $E$ for $W$: one takes a basis $\{e_i\}_{i\in I}$ of $V$, with a subset $\{e_j\}_{j\in J}$ that maps to a basis $\{\overline{e_j}\}_{j\in J}$ of $(V\otimes K)/W$; then one takes for $E$ the subfield of $K$ spanned by the coefficients of the remaining $\{\overline{e_i}\}_{i\in I\setminus J}$ when expressed as linear combinations of the $\overline{e_j}$. See also  \cite[Sec.~4.8] {EGAIV} for a discussion of fields of definition.

For an artinian $k$-algebra $A$, we will be interested in the field of definition of the Jacobson radical $\Jac(A_{\bark})$ as a submodule of $A_{\bark}$---the construction above implies its existence.

\begin{lemma}\label{lem:mfod}
 Let $k$ be a field and let $A$ be a finite-dimensional algebra over $k$.
 Let $\mathcal{J}:=\Jac(A_{\bark})$ denote the Jacobson radical of the base change $A_{\bark} = A\otimes_k \bark$.
 \begin{itemize}
  \item[(i)] If $A=K$ is a purely inseparable field extension of $k$, then as a $\mathcal{J}$ has minimal field of definition $K$. 
 \item[(ii)] More generally, the minimal field of definition of $\mathcal{J}$ is a finite purely inseparable extension $K/k$
 (and, in particular, it exists).\end{itemize}
 Let $J\subseteq A_{K}$ be such that $\mathcal{J} = J_\bark$.
\begin{itemize}\item[(iii)]  For any extension $E/K$ we have $\Jac(A_E)=J_E$.
\item[(iv)] The module $A_K$ is rigid if and only if $A_E$ is rigid for some field extension $E/K$ if and only if $A_E$ is rigid for all field extensions $E/K$.
\end{itemize}
\end{lemma}

\begin{proof}
  (i). Because $K$ is purely inseparable, for any algebraic extension $E/k$ the algebra $K\otimes_k E$ is local, since there is precisely one embedding $K\hookrightarrow \overline{E}=\bark$ and so there is just one possible quotient field: the compositum of $K$ and $E$. 
  Seen as an $E$-algebra this quotient field is $1$-dimensional precisely when $K\subseteq E$, hence $\dim \Jac(K\otimes_k \bark)=[K:k]-1$ and $K$ is the required minimal field of definition.
  
(ii). Since $A$ is finite-dimensional, by the Artin--Wedderburn theorem we have $$A_{\text{ss}}:=A_{k_s}/\Jac(A_{k_s})\cong\prod_{i=1}^r\Mat_{n_i}(k_i)$$ for $n_i\in\NN$ and each $k_i$ a finite purely inseparable field extension of $k_s$. By (i) we have that $k_i$ is the minimal field of definition for $\Jac(\Mat_{n_i}(k_i)_\bark)=\Mat_{n_i}(\Jac((k_i)_\bark))$ and we claim that the compositum $L$ of the $k_i$ is the finite purely inseparable extension of $k_s$ which is the minimal field of definition for $\Jac(A_{\text{ss}})_\bark$. Indeed we have $(A_{\text{ss}})_L/\Jac((A_{\text{ss}})_L)\cong\prod\Mat_{n_i}(L)$, which is an absolutely semisimple $L$-algebra expressed as a quotient of $A_L$ by a nilpotent ideal. Hence $\mathcal{J}$ is defined over $L$. Conversely, if $E$ does not contain some $k_i$, then there is some quotient algebra $\Mat_{n_i}(k_i)_E$ of $A_E$ whose Jacobson radical is $\Mat_{n_i}(\Jac(k_i\otimes_k E))$; by (i), $\Jac(k_i\otimes_k E)_\bark$ is a strict subalgebra of $\Jac(k_i\otimes_k \bark)$ and so $\mathcal{J}$ is not defined over $E$.

Since $\bark\cong \kp\otimes k_s$ where $\kp=k^{p^{-\infty}}$, there is a purely inseparable field extension $K'/k$ such that $K'\otimes_k k_s$ is an extension of $k_s$ containing $L$; by Galois descent this implies that the minimal field of definition $K$ over $k$ of $\Jac(A_L)$ (hence $\mathcal{J}$) is purely inseparable.

(iii). Keeping the notation already established, we now claim that $L = K\otimes_k k_s$.
First note that $K\otimes_k k_s$ is a field, since $K/k$ is purely inseparable and $k_s/k$ is separable.
Since $\mathcal{J}$ is $K$-defined, it is $(K\otimes_k k_s)$-defined, and hence $L\subseteq K\otimes_k k_s$ (we established that $L$ is the minimal field of definition of $\mathcal{J}$ over $k_s$).
But $L/k$ is an extension over which $\mathcal{J}$ is defined, so $K\subseteq L$, and the claim is proved.

Now we have observed that $A_L/J_L=(A_K/J_K)_L$ is absolutely semisimple.
If $E/K$ is any extension, we may identify all the fields so far encountered in an algebraic closure $\overline{E}$ of $E$.
The compositum $EL$ in $\overline{E}$ is generated by elements of the separable algebraic extension $L/K$, and hence $EL/E$ is separable.
Since $A_{EL}/J_{EL} = (A_L/J_L)_{EL}$ is semisimple, so is $A_E/J_E$, by an application of \cref{lem:semisimpledown}.

(iv). Since any artinian algebra is a direct product of local artinian algebras whose factors are the indecomposable summands of $A$ as a left $A$-module, we may assume it is local, since rigidity of a module is predicated on its indecomposable summands. Let $E/K$ be any field extension.
Since $A_E$ is artinian, the radical series for $A_E$ is formed by taking the powers of $\Jac(A_E)$, and the socle series by taking the annihilators of those powers, see \Cref{sec:artin}.
Part (iii) says that $\Jac(A_E) = J_E$, so the powers of $\Jac(A_E)$ are the base changes to $E$ of the powers of $J$.
Then the annihilators of the powers of $\Jac(A_E)$ are the base changes to $E$ of the annihilators of the powers of $J$ as well: certainly we have an inclusion $\Ann(J^i)_E \subseteq \Ann(J_E^i)$, and then a consideration of dimension gives equality.  
Hence, if the socle and radical series coincide over $E$, they already coincide over $K$, and vice versa.
\end{proof}

\begin{remark}\label{mfodsep}
Let us underline the aspect of the construction of $K$ in the proof which shows that the minimal field of definition of $\mathcal{J}$ commutes with separable extensions: this is used at the start of the proof of part (iii) above. 
Suppose $J\subseteq A_K$ is such that $J_\bark=\mathcal{J}$.
If $E/k$ is some separable extension,  $E':=K\otimes_k E$ is a field since $K/k$ is purely inseparable and $E/k$ is separable, and $\mathcal{J}$ is $E'$-defined via the ideal $J_{E'}$ of $A_{E'}$.
Thus the minimal field of definition over $E$ of $\mathcal{J}$ is contained in $E'$.
On the other hand, if $L/E$ is any extension over which $\mathcal{J}$ is defined, then $L/k$ is also an extension over which $\mathcal{J}$ is defined, so $L$ contains $K$.
Hence $L$ contains $E'$ and $E'$ is the minimal field of definition of $\mathcal{J}$ over $E$. 
\end{remark}

\subsection{Two Morita equivalences}\label{sec:morita}
Recall that $k$ is a commutative unital ring. Denote the comultiplication on $k[G]$ by $\Delta_G$ and the surjective algebra map $\epsilon_G:k[G]\to k$ that corresponds to `evaluation at the identity point'. From this, one can define an algebra structure on $k[G]^*:=\Hom_k(k[G],k)$ as follows. For $\mu,\nu\in k[G]^*$, we define $\mu\cdot\nu$ as $(\mu\otimes \nu)\circ \Delta_G$; more explicitly, if $\Delta_G(f)=\sum g_i\otimes h_i$ then \[(\mu\cdot\nu)(f)=\sum \mu(g_i)\otimes\nu(h_i).\] Then one checks from the Hopf algebra axioms that this makes $k[G]^*$ into an associative $k$-algebra with $\epsilon_G$ its unit---see \cite[I.7.7]{Jan03}. Furthermore, a $G$-module $M$ becomes naturally a $k[G]^*$-module: if $\Delta_M$ denotes the comodule map, then $\mu$ acts on $M$ by $(1\otimes \mu)\circ\Delta_M$, or more explicitly, if $\Delta_M(m)=\sum m_i\otimes f_i$, then $\mu(m):=\sum m_i\mu(f_i)$; see \cite[I.7.11]{Jan03} for more details.

Now suppose that $k[G]$ is a flat and projective $k$-module---immediate when $k$ is a field. Since $k[G]=k\cdot 1 \oplus I_1$ for $I_1$ the functions vanishing at the identity point, we have that $I_1$ is also flat and projective and $k[G]^*=k\cdot \epsilon_G \oplus I_1^*$. Under these hypotheses every $G$-module is locally finite and in fact $k[G]^*\cdot m=kGm$, where $kGm$ denotes the smallest $G$-submodule of $M$ containing $m$; see \cite[Prop.~3]{Sesh} and its proof. Also, by projectivity, we may apply the dual basis lemma, \cite[2.9]{Lam2} to find an indexing set $\mathbb{I}$ and a family of pairs $\{(f_i,\mu_i)\mid i\in \mathbb{I}\}\subset I_1\times I_1^*$, such that for any $f\in I_1$, $\mu_i(f)\neq 0$ for only finitely many $i$ and \begin{equation}
  f=\sum_{i\in\mathbb{I}}\mu_i(f)f_i.\label{dbeq}
\end{equation} For convenience let us add a new element $0$ to $\mathbb{I}$ with $f_0=1$ and $\mu_0=\epsilon_G$. Then \cref{dbeq} holds with $f\in k[G]$. (Clearly $\{f_i\mid i\in\mathbb{I}\}$ is now a generating set of $k[G]$ as a $k$-module.) Let $\M$ be the subalgebra of $k[G]^*$ generated by the $\mu_i$. 

\begin{lemma}\label{lem:GmodMmod}
  With the above hypotheses, suppose $M$ is a $G$-module and $m\in M$. Then $\M m=kGm$. Hence a $k$-submodule $N$ of $M$ is an $\M$-submodule if and only if it is a $G$-submodule.
  \end{lemma}
  \begin{proof}
  It is convenient to first show how to deduce the second sentence given the first. 
  First, if $N$ is a $G$-submodule of $M$ then the flatness of $G$ implies $\Delta_M(N)\subseteq N\otimes k[G]\subseteq M\otimes k[G]$, and so $\mu(N)\subseteq N\otimes \mu(k[G])\subseteq N$ for any $\mu\in k[G]^*$, showing that $N$ is an $\M$-submodule
  (note that this does not rely on the first sentence of the statement). 
  Conversely, if $N$ is an $\M$-submodule then the first sentence of the statement implies $$N=\sum_{m\in N} \M\cdot m=\sum_{m\in N}kGm,$$ and the latter is a sum of $G$-submodules, hence a $G$-submodule.

For the first statement, it is shown in \cite[I.2.13]{Jan03} that $\Delta_M(m)\in kGm\otimes k[G]$ and hence we may write $\Delta_M(m) = \sum_{j\in\mathbb{J}} m_j\otimes g_j$ for some finite indexing set $\mathbb{J}$ and with each $m_j\in kGm$.
When we do this, we have $kGm=\sum_{j\in\mathbb{J}} km_j$, again by \cite[I.2.13]{Jan03}. 
Since for each $g_j$, there are only finitely many $\mu_i$ which do not vanish on $g_j$, and since $\mathbb{J}$ is finite, we can find a finite subset $\mathbb{I}'\subseteq\mathbb{I}$ containing $0$ such that $\sum_{i\in\mathbb{I}'}f_i\mu_i$ is the identity map on the $k$-submodule of $k[G]$ generated by the $g_j$. 
    
Let $M'=\sum_{i\in \mathbb{I}'}k\mu_i(m)$ and claim $M'=kGm$. 
Since $kGm$ is a $G$-submodule containing $m$, it is an $\M$-submodule by the start of the proof; in particular, $\mu_i(m)\in kGm$ for each $i$, so $M'\subseteq kGm$.   

For the other inclusion, we first show that $\Delta_m(m)\in M'\otimes k[G]$.
For this, note that for each $i\in\mathbb{I}'$ we have $\mu_i(m) = \sum_{j\in \mathbb{J}} \mu_i(g_j) m_j$ and thus 
    \[\sum_{j\in \mathbb{J}} m_j \mu_i(g_j)\otimes 1= \sum_{j\in\mathbb{J}} m_j \otimes \mu_i(g_j)\] is in $M'\otimes k[G].$
    Multiplying by $1\otimes f_i$ we get 
    $\sum_{j\in \mathbb{J}} m_j\otimes f_i\mu_i(g_i) \in M'\otimes k[G]$, and now summing up over $i\in \mathbb{I}'$ we get 
    \[\sum_{j\in \mathbb{J}} m_j\otimes \left(\sum_{i\in \mathbb{I}'} f_i\mu_i(g_j)\right) = \sum_{j\in \mathbb{J}} m_j\otimes g_j \in M'\otimes k[G],\]
  as required.

Thus, we may write $\Delta_M(m) =   \sum n_i\otimes h_i$ for some $n_i\in M'$ and $h_i\in k[G]$.
If we let $N = \sum kn_i$, then it follows from \cite[I.2.13(1)]{Jan03} that $kGm\subseteq N$.
But then we have $kGm\subseteq M'$ and we are done.
\end{proof}
If $\{(f_i,\mu_i)\}$ is a dual basis in the sense described above, then it remains so after flat base change. Hence the construction of $\M$ from $k[G]^*$ commutes with flat base change. 
Therefore we may replace $(G,M,\M)$ by $(G_E,M_E,\M_E)$ in the theorem and we get the following:

  \begin{cor}
    Suppose $M$ is a $G$-module and $E$ a flat $k$-algebra. Then the $G_{E}$-submodules of $M_{E}$ are just the $\M_{E}$-submodules of $M_{E}$. Moreover if $R$ denotes the image of $\M$ in $\End_k(M)$, then the $G_{E}$-submodules of $M_{E}$ are just the $R_{E}$-submodules of $M_{E}$.
  \end{cor}

We apply the corollary to the case where $M$ is a simple $G$-module which is projective over $k$. By local finiteness, we have $M$ is finitely generated and so $M^*$ is too \cite[2.11]{Lam2}. Furthermore, as $\End_k(M)\cong M^*\otimes_k M$ \cite[Ex.~\S2.20]{Lam2} so also $\End_k(M)$ is finitely generated and projective. Now Schur's Lemma \ref{lem:schur} tells us that $\End_G(M)\cong \End_R(M)=:D$ is a division algebra over $k$, and it is finitely-generated as a $k$-module, hence artinian. As $R$ is left primitive (i.e.~has a faithful left module), it must be simple---\!\!\cite[11.7]{Lam1}. 

An $R$-module $P$ is said to be a \emph{left generator} for $R$ if $\Hom_R(P,?)$ is a faithful functor from the category of left $R$-modules to the category of abelian groups---\!\!\cite[18.7]{Lam2}. If in addition $P$ is finitely generated and projective it is called a \emph{progenerator} for $R$. 
As $R$ is simple, the category of left $R$-modules is semisimple and so any nonzero module is a generator; 
thus $M$ is a progenerator of $R$. Since $D$ is the centraliser of $R$ in $\End_k(M)$, we get that $M$ is a right $D$-module and $M$ is an $(R,D)$-bimodule that is faithfully balanced \cite[18.19,18.21]{Lam2}; this is to say the maps $R\to \End_D(M)$ and $D\to \End_R(M)$ are both ring isomorphisms---so $R$ is also the centraliser in $\End_k(M)$ of $D$. In particular, $R$ and $D$ are Morita equivalent---\!\!\cite[18.33]{Lam2}.

Now, under flat base change through $k\to E$,  we have $D_E\cong \End_{R_E}(M_E)$.
We also have that $M_E$ is a progenerator for $R_E$: it is finitely-generated projective since $M$ is and $E$ is flat; and it is a generator
by the characterisation in \cite[18.8(3)]{Lam2} applied to $M$ (resp. $M_E$) and $R$ (resp. $R_E$). 
Thus $M_E$ is a faithfully balanced $(R_E,D_E)$-bimodule, so  $R_E$ and $D_E$ are again Morita equivalent and
 \cite[18.44]{Lam2} furnishes us with:
\begin{prop}\label{prop:REDE}
The  lattice of (left) $R_E$-submodules of $M_E$ is isomorphic to the lattice of left ideals of $D_E$.
\end{prop}

We can push this analysis one step further.
Let $Z :=Z(D)$ be the centre of the division ring $D$.
Considering $D$ as a left $D$-module, multiplication on the right by elements of $D$ gives an identification $\End_D(D) \cong D$, 
and the centraliser of $D$ is just the centre $Z$.
It is clear that we get another Morita equivalence with $D$ itself as the progenerator this time, and we deduce:
\begin{cor}\label{cor:DEZE}
The  lattice of left ideals of $D_E$ is isomorphic to the lattice of ideals of $Z_E$.
\end{cor}

We make one final observation about homomorphisms of $G$-modules.

\begin{lemma}\label{lem:GhomMhom}
Suppose that $M$ is a finitely-generated projective $G$-module and $N$ is another $G$-module.
Then a $k$-linear map $M\to N$ is a $G$-module homomorphism if and only if it is an $\M$-module homomorphism.
\end{lemma}

\begin{proof}
Since $M$ is finitely generated and projective, we can give the space of $k$-linear maps $\Hom(M,N)$ the structure of a $G$-module and identify
\[
\Hom_G(M,N) \cong \Hom(M,N)^G,
\]
where the RHS is the $G$-fixed points, see \cite[I.2.10(6)]{Jan03}.
But the $G$-fixed points coincide with the $\M$-fixed points, by \Cref{lem:GmodMmod},
so we're done.
\end{proof}

\begin{center}\framebox{For the rest of the paper, $k$ will denote a field.}\end{center}

In particular, since $k$ is a field, then $G$ satisfies the hypotheses at the beginning of the section; so the conclusions of \cref{prop:REDE} and  \cref{cor:DEZE} both hold. The following result  also applies to algebras over a field and enables the use of the Conrad--Prasad classification.

\begin{lemma}\label{directproduct}Suppose that $A$ and $B$ are two  associative unital $k$-algebras and put $C:=A\otimes_k B$, where $A$ and $B$ are considered as commuting subalgebras of $C$ via $A\cong A\otimes 1$ and $B\cong 1\otimes B$. 
  Let $W$ be a simple $C$-module which is finite-dimensional over $k$. Then the following hold:
 \begin{enumerate}
   \item $W|_{A}$ is an  isotypic semisimple $A$-module, say $U^r$ for a simple $A$-module $U$, and $r\in\NN$; similarly let $W|_{B}\cong V^s$ for a simple $B$-module $V$. There is a surjective $C$-module homomorphism $\psi:U\otimes V\to W$.\label{ds3}
 \end{enumerate}
   Let $D:=\End_A(U)$, $E:=\End_B(V)$ and $F:=\End_C(W)$. Then:
   \begin{enumerate}
   \setcounter{enumi}{1}
   \item $\End_C(U\otimes V)\cong D\otimes E$;\label{ds1}
   \item $Z(D\otimes E)\cong Z(D)\otimes Z(E)$;\label{ds2}
   \item $C/\Ann_C(W)$ is a simple $k$-subalgebra of $\End_k(W)$ generated by the images of $A\otimes 1$ and $1\otimes B$; \label{ds4}
   \item if $k=k_s$, then $D = Z(D)$, $E = Z(E)$ and $F = Z(F)$ are finite extensions of $k$, necessarily purely inseparable, and $F$ is the compositum of $D$ and $E$;\label{ds5}
   \item for arbitrary $k$, all of the following fields coincide\label{ds6}:
   \begin{enumerate}\item the minimal fields of definition of $\Jac(F_\bark)$ and $\Jac(Z(F)_\bark)$; 
     \item the compositum of the minimal fields of definition of $\Jac(D_\bark)$ and $\Jac(E_\bark)$;
     \item the compositum of the minimal fields of definition of $\Jac(Z(D_\bark))$ and $\Jac(Z(E_\bark))$.\label{dstest}
   \end{enumerate}
 \end{enumerate}
\end{lemma}
\begin{proof} 
 \cref{ds3} is \cite[VIII, ~\S12.1, Prop.~2]{BourbAlg8}, but we nutshell the details. One takes a simple $A$-submodule $U$ of $W$ and considers $X:=\Hom_A(U,W)$, which becomes a $B$-module via $(b\circ\phi)(u)=b(\phi(u))$. Let $\varphi:V\to X$ be the embedding into $X$ of some simple $B$-submodule. Then define $\psi:U\otimes V\to W$ by $\psi(u\otimes v)\mapsto \varphi(v)(u)$. Then one checks $\psi$ is a non-zero $C$-module map, and the isotypicity of $W|_{A}$ follows from that of $(U\otimes V)|_A\cong U^{\dim V}$.
 
Now \cref{ds1} and \cref{ds2} are \cite[VIII, \S12.5, Prop.~5(a)] {BourbAlg8}. By Morita theory---see \cref{prop:REDE}---the fact that $W$ is simple implies that $C/\Ann_C(W)$ is simple, giving \cref{ds4}. 

Take $k=k_s$. 
In this case, $D$ is a centrally finite division algebra, and $Z(D)$ is a finite extension of $k = k_s$, hence purely inseparable over $k$ and separably closed.
By \cite[15.12]{Lam1}, there is a maximal subfield of $D$ which is a separable extension of $Z(D)$; but $Z(D)$ is separably closed and hence itself a maximal subfield of $D$.
Now \cite[15.7]{Lam1} implies that $D = Z(D)$ is a field as claimed. 
The same argument goes to show $E=Z(E)$ and $F=Z(F)$.   
Now $D\otimes_k E$ is a tensor product of purely inseparable extension fields and so it is local, and so $F$ must identify with its quotient field, which is the compositum of $D$ and $E$, completing the proof of \cref{ds5}.

Since the minimal fields of definition of the Jacobson radicals commute with separable extension (Remark \ref{mfodsep}), we may base change to $k_s$, whereupon $W$ is still semisimple by \cref{lem:sepext}, and argue with each simple submodule individually. This gives us a collection of division rings $(D_i,E_i,F_i)$ and the minimal fields of definition of the Jacobson radical of $F_{k_s}:=\prod F_i$ is the compositum of all the $Z(D_i)$ and $Z(E_i)$---\cref{lem:mfod}(i). This proves \cref{ds6}.
\end{proof}

Suppose that $H_1$ and $H_2$ are algebraic $k$-groups and put $J:=H_1\times H_2$.
Since we are working over a field, we can form the $k$-algebras $\M_1$, $\M_2$ and $\M$ respectively from dual bases $\{(f_i,\mu_i)\}$, $\{(g_j,\nu_j)\}$ and $\{f_i\otimes g_j,\mu_i\otimes\nu_j\}$ respectively. 
Using \Cref{lem:GmodMmod}, \Cref{lem:GhomMhom}, and the previous result 
in the case $A=\M_1$, $B=\M_2$, and $C=\M$, we get the following.

\begin{cor}\label{redtofactors}
  Suppose that $V$ is a simple $J$-module. 
Then the restriction of $V$ to $H_i$ decomposes as a direct sum of copies of a single simple module $W_i$ for $i=1,2$.
Denote by 
$$D:=\End_{J}(V),\ D_1:=\End_{H_1}(W_1),\ D_2:=\End_{H_2}(W_2),$$ the respective division rings. 
Then if $K$, $K_1$, $K_2$ are the respective minimal fields of definition of their geometric Jacobson radicals (or equivalently, the geometric Jacobson radicals of the centres $Z(D)$, $Z(D_1)$ and $Z(D_2)$), then $K$ is the compositum of $K_1$ and $K_2$.
\end{cor}

From all of this, we get
\begin{cor}\label{cor:GEDE}
Let $M$ be a simple $G$-module, and $E/k$ an extension of fields.
The $G_E$-module $M_E$ is rigid if and only if the commutative algebra $Z_E$ is rigid (as a module for itself). 
\end{cor}

\begin{proof}

The hypotheses on $k$ and $G$ ensure that all the preceding results in this section hold, and in particular we can use  \Cref{prop:REDE} to relate the $G_E$-module structure on $M_E$ to the $R_E$-module structure.
But then we see that $M_E$ is rigid as an $R_E$-module if and only if $Z_E$ is rigid, by \Cref{cor:DEZE}
applied to this case.
\end{proof}

Thus, to complete the proof of \cref{thm:mainone}, we need to show the existence of a purely inseparable extension $K/k$ such that $Z_E$ is rigid for every extension $E/K$. 
This is achieved in the next section, where we show that we can take $K$ to be the minimal field of definition of $\Jac(Z_\bark)$.

Finally in this section we give some further generalities about semisimple $k$-algebras, which will help us when we are describing
the endomorphism rings of simple modules in Part \ref{partii}.

\begin{defn}\label{defn:psplitting}
  Let $A$ be a finite-dimensional semisimple $k$-algebra over a field of characteristic $p$. 
  We say $A$ is \emph{$p$-split} if $A$ is $k$-isomorphic to a direct product of matrix rings over purely inseparable field extensions of $k$.

  If $A$ is a $k$-algebra and there exists a field extension $E/k$ such that $A_E:=A\otimes_kE$ is $p$-split as an $E$-algebra, then we say that $A$ is \emph{$p$-splittable}. 
  If $[E:k]$ is minimal such that $A_E$ is $p$-split, then we say $E$ is a \emph{$p$-splitting field for $A$}.

  If $A$ is a simple $k$-algebra whose centre $Z(A)$ is a purely inseparable extension of $k$, then we say that $A$ is a \emph{$p$-central simple algebra} (or \emph{$p$-CSA}). 
  If $A$ is also a division algebra then we say $A$ is a \emph{$p$-division algebra}.
\end{defn}

We aim to show:

\begin{theorem}
 Let $A$ be a finite-dimensional semisimple $k$-algebra and $Z:=Z(A)$ its centre. Then: 
\begin{itemize}
\item[(i)] the algebra $A$ is $p$-split by a field $\mathscr{E}$ that is finite and Galois over $k$;
\item[(ii)] if $A$ is simple then any $p$-splitting field $E$ for $A$ contains the normal closure $F$ of the separable part $Z_\sep$ of the field extension $Z/k$. If $\ell:=[Z_\sep:k]$ then $$Z(A_F)\cong \underbrace{ZF\times\dots\times ZF}_{\ell-\text{times}},\qquad\text{and}\qquad A_F\cong A_1\times\dots\times A_\ell$$ as a direct product of $p$-CSAs. 
\end{itemize}
\end{theorem}

\begin{proof}
  The existence of a field extension which $p$-splits $A$ is obvious: taking $E=k_s$ one gets that $A_E$ is still semisimple and is therefore still a product of matrix rings over division $k_s$-algebras by Artin--Wedderburn. But the Brauer group of $k_s$ is trivial; so all these division algebras are purely inseparable field extensions of $k_s$.
  Now $k_s/k$ is a filtered union of finite Galois extensions, so this $p$-splitting must already happen over some finite Galois extension of $k$, proving (i) (see \cite[Sec. 8]{EGAIV} for this sort of argument in great generality).
    
For (ii), we first prove the result for a division $k$-algebra $D$; the general result follows quickly from Artin-Wedderburn.
Let $Z=Z(D)$, let $Z_\sep$ be the separable part of this field extension of $k$, and
let $E/k$ be some field extension that $p$-splits $D$; so $D_E:=D\otimes_k E$ is a direct product of matrix algebras over a set  of finite purely inseparable field extensions of $E$; and its centre $Z(D_E)$ is isomorphic to the direct product of these fields. Take any $x\in Z_\sep$ with minimal polynomial $f$ over $k$. Then $k[x]/f$ identifies with a subfield of $Z_\sep$ and therefore $(k[x]/f)\otimes_k E\cong E[x]/f$ identifies with an $E$-subalgebra of $Z_E$. 
  But since $f$ is separable, $E[x]/f$ is also isomorphic to a direct product of separable extensions of $E$, and since $Z_E$ is a product of purely \emph{in}separable extensions of $E$ we see that $E$ contains all the roots of $f$.
  Thus $E$ contains the normal closure $F$ of $Z_\sep$ as claimed. This proves the first part of (ii).

  For the second, note that we have 
  $$
  Z_F= Z\otimes_k F \cong Z\otimes_{Z_\sep} (Z_\sep\otimes_k F).
  $$ 
  Writing $Z_\sep = k[x]/f$ for some separable polynomial $f$ and using the fact that the roots of $f$ all lie in $F$ by the previous paragraph, we see that $Z_\sep\otimes_k F$ is isomorphic to a product of $\ell=[Z_\sep:k]$ copies of $F$, conjugate under the Galois group of $F/k$. 
  Since $Z/Z_\sep$ is purely inseparable and $F/Z_\sep$ is separable, $Z$ and $F$ are linearly disjoint; so $Z_F$ is isomorphic to $\ell$ copies of the compositum $ZF:=(Z\otimes_{Z_\sep}F)$, and these fields are still conjugate under $\Gal(F/k)$. 
  Now $D_F$ is semisimple because $F/k$ was separable; so $$D_F\cong A_1\times \dots\times A_\ell$$ where the $A_i$ are $p$-CSAs that are all conjugate under $\Gal(F/k)$: 
 this decomposition follows from the decomposition of $1\in Z_F$ into central primitive idempotents corresponding to the $\ell$ Galois-conjugate factor fields of $Z_F$, see \cite[7.22]{Lam2}. This proves (ii) in our special case; and the general case of (ii) is also clear as $A\cong \Mat_r(D)$ for some division $k$-algebra $D$.  
\end{proof}

\begin{remark}\label{rem:noncan}
While the field $F$ in part (ii) of the theorem is uniquely determined, the field $E$ is not, even given the specification of \Cref{defn:psplitting} that $[E:k]$ should be minimal. 
For example when $A=\Q\oplus i\Q\oplus j\Q\oplus k\Q$ is the $\mathbb{Q}$-division algebra of Hamiltonian rational quaternions, then it is ($p$-)split by many quadratic field extensions. Moreover,  following Amitsur's original construction, there are by now many examples of finite-dimensional central $k$-division algebras $D$ which are \emph{non-crossed}; that is to say that there is no maximal subfield of $D$ that is Galois. From this it also follows that $D$ cannot be constructed from a cocycle using Noether's method---see the introduction of \cite{Han04} for an overview. In any case, the point to make here is that the field $\mathscr{E}$ is highly non-canonical in general.

  In the scenario of interest to us---namely when $D=\End_G(V)$ for $V$ a finite-dimensional $G$-module---we will show by contrast that there is a \emph{unique} $p$-splitting field $\mathscr{E}$, and that is Galois over $k$; see \cref{thm:torussplit}.
\end{remark}

\subsection{On the rigidity of finite-dimensional algebras}\label{sec:rigidity}

\begin{theorem}\label{thm:Drigid}
Suppose that $R$ is a finite-dimensional simple $k$-algebra (though not necessarily central simple).
Let $k'/k$ be the minimal field of definition of $\Jac(R_{\bark})$. Then $R_{k'}$ is absolutely rigid. 
\end{theorem}

Note that the proof of this theorem immediately reduces to the case that $R$ is a field:
since $R$ is simple, the Artin-Wedderburn Theorem says that $R\cong M_n(D)$ for some $D$, where $D$ is a finite-dimensional division $k$-algebra.
For any field extension $E/k$, we have $R_E \cong M_n(D_E)$, and the ideal structure of $R_E$ is therefore identical with that of $D_E$---\cite[3.1]{Lam1}.
But now we may replace $D$ with its centre $Z$, as in \Cref{cor:DEZE}.
Since $R$ is finite-dimensional, $Z/k$ is some finite extension of $k$.
We can also see that the minimal fields of definition of $\Jac(R_\bark)$, $\Jac(D_\bark)$ and $\Jac(Z_\bark)$ all coincide.

\subsubsection{Generalities on tensor products of fields}\label{sec:tensor}
The literature already contains a number of results about the tensor product $A:=K\otimes_k E$ of general field extensions $K/k$ and $E/k$. 
For example, it is a result of Grothendieck, in generalised form by Sharp \cite{Sha77} that $\dim(K\otimes_k E)=\min(\mathrm{tr.deg}(K/k),\mathrm{tr.deg}(E/k))$. 
Furthermore Grothendieck proved in \cite[Lem.~6.7.1.1]{EGAIV} that if one of the extensions has finite type, then $A$ is Cohen-Macauley. This has been generalised in at least two directions: in \cite[Lem.~2.2]{BK02} weakening the hypotheses to demanding $A$ be noetherian; and in \cite[I.Thm.~2]{watanabe} strengthening the conclusion to saying $A$ is  Gorenstein. 

A \emph{simply truncated polynomial algebra (STP algebra)} over a field $k$ is an algebra of the form
$$
A=k[X_1,\ldots,X_n]/(X_1^{a_1},\ldots,X_n^{a_n})
$$
with $a_1\geq a_2 \geq \cdots\geq a_n$ \cite[Ch. 1]{rasala}.
Let $x_i$ denote the image of $X_i$ in $A$.
Then $A$ is a local ring, with maximal ideal $\m$ generated by the $x_i$.
It is clear that if $A$ is an STP algebra then, for any field extension $E/k$, $A_E = A\otimes_k E$ is an STP algebra over $E$.

\begin{lemma}\label{rigidSTP}
An STP algebra $A$ is a rigid local Gorenstein algebra, and hence has a symmetric Hilbert function.
\end{lemma}

\begin{proof}
Since $A$ is local, it is indecomposable and thus we need to check \eqref{rigidann}.
For a tuple $\beta = (b_1,\ldots,b_n)$ of non-negative integers, we let $x^\beta := \prod_{i=1}^n x_i^{b_i}$,
and for another such tuple $\beta' = (b_1',\ldots,b_n')$ we say $\beta\leq \beta'$ when $b_i\leq b_i'$ for every $i$.  
Note that for any $\beta = (b_1,\ldots,b_n)$, we have $x^\beta = 0$ if and only if there exists some $i$ with $b_i\geq a_i$.
Let $\gamma = (a_1-1,\ldots,a_n-1)$, and let $n = \left(\sum_{j=1}^n (a_j-1)\right) + 1$.
Since $x^\gamma\in\m^{n-1}$ we have $\m^n = \{0\}$ but $\m^{n-1}\neq \{0\}$.
We need to show that $\Ann(\m^i) = \m^{n-i}$ for each $1\leq i\leq n-1$.
It is clear that for any $1\leq i\leq n-1$, $\m^{n-i}$ annihilates $\m^i$.
On the other hand, given $1\leq i\leq n-1$, the ideal $\m^{n-i-1}$ is generated by the $x^\alpha$ with $\alpha=(\alpha_1,\ldots,\alpha_n) \leq \gamma$ and $\sum_{j=1}^n \alpha_j = n-i-1$.
Given any such $\alpha$, let $\beta = \gamma-\alpha$.
Then $x^\alpha x^\beta = x^\gamma \neq 0$, and $\sum_{j=1}^n \beta_j = i$, so $x^\alpha \not\in \Ann(\m^i)$.
Thus $\Ann(\m^i) = \m^{n-i}$, and we see that $A$ is rigid.

The final power $\m^{n-1} = \Ann(\m) = \Soc(A)$ is generated by the element $x^\gamma$, and so is simple as an $A$-module, so $A$ is Gorenstein.
The symmetry of the Hilbert function of $A$ follows from \cite[\S70]{Mac}, as explained in Remark \ref{rem:rigidegs}.
\end{proof}

Now suppose $Z/k$ is some finite extension, and let $M\subseteq Z$ be the separable part of the extension.
From \cite[Ch.~2, Thm.~6]{rasala} we learn that there is some finite (normal) extension $E/Z$ such that $Z\otimes_k E$ is a sum of STP $E$-algebras, and the summands are all isomorphic (via Galois automorphisms) to $Z\otimes_M E$.
In particular, $Z\otimes_k \bark$ is a sum of STP $\bark$-algebras, and all the summands are isomorphic as rings.

\begin{prop}\label{stp}
Let $k'$ be the minimal field of definition of $\Jac(Z_\bark)$.
Then $Z_{k'} := Z\otimes_k k'$ is absolutely rigid.
Thus for all extensions $E/k'$, the Hilbert functions $H(Z_E)$ coincide, and are symmetric about the middle degree term.
\end{prop}

\begin{proof}
By the results of \cite{rasala} above, $Z_\bark$ is a sum of isomorphic copies of the STP $\bark$-algebra $Z\otimes_M \bark$;
these will be the blocks of $Z_\bark$.
Since by \cref{rigidSTP} each of these is rigid, so is $Z_\bark$. 
In fact, since the blocks are all isomorphic as rings, we are in the sitiuation where the socle/radical series of $Z_\bark$ is given by powers of its Jacobson radical.
In any case, we can now apply 
\Cref{lem:mfod}(iii) with $A=Z$ to deduce rigidity of $Z_E$ for all extensions $E/k'$.

The final statement about the Hilbert functions follows since the socle and radical series for any $Z_E$ are just the base changes of those for $Z_{k'}$, and 
so all the Hilbert series coincide.
The dimensions of quotients for that series can therefore be calculated over $\bark$, and since $Z_\bark$ is the sum of isomorphic copies of an STP, \Cref{rigidSTP} gives us the symmetry result.
\end{proof}

The proposition above completes the proof of \cref{thm:Drigid}, recalling the observations at the start of this section. It also completes the proof of \Cref{thm:mainone} 
using the Morita equivalences of \cref{sec:morita}:
if $V$ is a simple $G$-module, then $V_{k_V}$ is absolutely rigid,
where $k_V$ is the minimal field of definition of $\Jac(\End_G(V)_\bark)$.

\subsubsection{Example: a non-rigid tensor product of fields}
We give an example of a tensor product of two finite purely inseparable field extensions $K/k$, $E/k$ whose regular module is not rigid,
which shows that \cref{thm:Drigid} does not hold in full generality without the extra hypothesis involving the minimal field of definition $k'$.

\begin{example}\label{ex:finally}
Let $k=\F_2(a,b,c,d)$ where $a,b,c,d$ are algebraically independent transcendental elements.
Consider the following purely inseparable extensions $K/k$ and $E/k$
\begin{align*}
  K&:=k(\underbrace{a^{1/16}+b^{1/4}}_\beta,\underbrace{a^{1/8}+c^{1/4}}_\gamma,\underbrace{a^{3/16}+d^{1/4}}_\delta) = k(\beta,\gamma,\delta)\\
  E&:=k(\underbrace{a^{1/16}}_{f_1},\underbrace{a^{3/16}+a^{1/8}b^{1/4}+a^{1/16}c^{1/4}+d^{1/4}}_{f_2}) = k(f_1,f_2),
\end{align*}
and let $A = K\otimes_k E$ with maximal ideal $\m_A$.

Note that $\gamma^4 = a^{1/2} + c = \beta^8+b^2+c$, and 
$\delta^4 = a^{1/4} + d = (\beta^4 + b)(\gamma^4+c)$,
so $[K:k] = 16\times 4\times 4 = 2^8$.
We also have $[E:k]= 16\times 4 = 2^6$, and hence $\dim_k(A) = 2^{14} = 16,384$.
The compositum of $K$ and $E$ is the field 
$$
KE = k(a^{1/16},b^{1/4},c^{1/4},d^{1/4}) \cong A/\m_A
$$ 
of degree $2^{10}$ over $k$, so the maximal ideal $\m$ has dimension $2^{14} - 2^{10} = 2^{10}\times 15$ over $k$.
Thus $A$ has $16$ composition factors isomorphic to $KE$, with $15$ of them coming from $\m$.

Viewing $A$ as a $K$-algebra through multiplication in the first factor, $A$ has a $K$-basis consisting of
the $2^6$ elements $1\otimes f_1^if_2^j$ where $0\leq i\leq 15$ and $0\leq j\leq 3$.
An element $x = \sum_{i,j} e_{ij}\otimes f_1^if_2^j$ lies in $\m_A$ if and only if $\sum_{i,j} e_{ij}f_1^if_2^j = 0$ in $KE$.
In other words, the elements of $\m_A$ correspond to $K$-linear dependences between the $f_1^if_2^j$.
Note that $f_1^4=a^{1/4}=\beta^4+b\in K$, and $f_2 = \beta f_1^2 + \gamma f_1 + \delta  \in K(f_1)$, so we see that 
the elements
$$
m_1:=1\otimes f_1^4+f_1^4\otimes 1 \qquad \textrm{ and } \qquad m_2:=1\otimes f_2 + \beta\otimes f_1^2+\gamma\otimes f_1+\delta\otimes 1
$$
lie in $\m_A$.
Further, given a general element $\sum_{i,j} e_{ij}\otimes f_1^if_2^j$ in $\m_A$, one can eliminate any occurrences of $f_2$ in the second tensor factors using multiples of $f_2$, and then further eliminate any powers of $f_1^4$ using multiples of $f_1$.
But the elements $1,f_1,f_1^2,f_1^3$ are $K$-linearly independent, and so we conclude that $m_1$ and $m_2$ generate $\m_A$.

The first power of $m_1$ that is $0$ is $m_1^4=0$ and the first power of $m_2$ that is $0$ is $m_2^8= 0$.
Also,
\begin{align*}
m_2^4 
&= 1\otimes f_2^4 + \beta^4\otimes f_1^8+\gamma^4\otimes f_1^4+\delta^4\otimes 1 \\
&= 1\otimes (a^{3/4} + a^{1/2}b + a^{1/4}c + d) + (a^{1/4} + b)\otimes a^{1/2} + (a^{1/2} + c) \otimes a^{1/4} + (a^{3/4} + d)\otimes 1 \\
&= 1\otimes a^{3/4} + a^{1/4}\otimes a^{1/2} + a^{1/2}\otimes a^{1/4} + a^{3/4}\otimes 1 \\
&= m_1^3.
\end{align*}
Thus we have the following powers of $\m_A$:
\begin{center}\begin{minipage}[t]{0.4\textwidth}\begin{align*}
\m_A &=  (  m_1, m_2 ) , \\
\m_A^2 &=  (  m_1^2, m_1m_2, m_2^2 ) ,\\
\m_A^3 &=  (  m_1^2m_2,m_1m_2^2,m_2^3 ) ,\\
\m_A^4 &=  (  m_1^2m_2^2,m_1m_2^3,m_2^4 ) ,\\
\end{align*}\end{minipage}
\begin{minipage}[t]{0.4\textwidth}\begin{align*}
  \m_A^5 &=  (  m_1^2m_2^3,m_2^5 ) ,\\
\m_A^6 &=  (  m_2^6 ) ,\\
\m_A^7 &=  (  m_2^7 ) \\
\m_A^8&=0.\end{align*}
\end{minipage}\end{center}
So the first power of $\m_A$ which is zero is $\m_A^8$.
We can see that $m_1\in \Ann_A(\m_A^6)$, since $m_1m_2^6 = m_1^4m_2^2 = 0$, showing that $\Soc^6(A) = \Ann_A(\m_A^6)\neq \m_A^2 = \Rad^2(A)$.
Thus $A$ is not rigid.
\end{example}

\begin{remark}While  Example \ref{ex:finally} proves that tensor products of field extensions are not generally rigid, it can often happen ``by accident'', even when $E$ does not contain $K$. For example, if $K$ and $E$ are linearly disjoint over $k$, then $K\otimes_k E$ is a field; this happens for example if $E$ is separable and $K$ is purely inseparable. 
  
Or, suppose $K=k(f)$ is a purely inseparable simple extension. Then if we let $r$ be minimal such that $f^r\in E$, the maximal ideal $\m_A$ is the principal ideal generated by the element $x = 1\otimes f^r-f^r\otimes 1$, and rigidity follows easily.
 \end{remark}

The following example helps to motivate our chosen definition of rigidity and shows how base change of a field along a (non-normal) field extension can give rise to an algebra whose regular module has indecomposable modules of different Loewy lengths.
\begin{example}\label{rem:notbottom}
Suppose that $p=3$ and let $k = \F_3(t,u)$ be the field of rational functions in $t$ and $u$ over the finite field $\F_3$. 
Let $F = k(t^{1/6} + u^{1/3})$.
Then $F$ has degree $6$ over $k$, and is made up of a Galois extension $E/k$ of degree $2$, where $E = k(t^{1/2})$, together with a purely inseparable extension $F/E$ of degree $3$.
The field $E$ has another purely inseparable extension of degree $3$ which is abstractly isomorphic to $F$, namely $F^* = k(-t^{1/6}+u^{1/3})$.
Let $L = k(t^{1/6},u^{1/3})$ be the compositum of $F$ and $F^*$, so $L/k$ has degree $18$.
The nontrivial Galois automorphism $\gamma: t^{1/2}\mapsto -t^{1/2}$ of $E$ extends in an obvious way to an automorphism of $L$ which swaps the subfields $F$ and $F^*$; we denote this automorphism by $\gamma$ as well.

We claim that the regular module for $A:=F\otimes_k F$ has indecomposable summands of different Loewy lengths.
There are two maximal ideals $\m_1$ and $\m_2$ in $A$: we can realise $\m_1$ as the kernel of the map $x\otimes y\mapsto xy$ with image $F$, and $\m_2$ as the kernel of the map $x\otimes y\mapsto x\gamma(y)$ with image $L$.
Then 
$$
A\cong (F\otimes_E F) \oplus (F\otimes_E F^*).
$$
Both summands are rigid, however the first one has Loewy length $3$ and the second is isomorphic to the field $L$, so has Loewy length $1$.
 
Some further analysis in this example shows that the minimal field of definition of $\Jac(F_\bark)$ is the field $k':=k(t^{1/3},u^{1/3})$, the subfield of $\gamma$-fixed points in $L$. 
Certainly $\Jac(F_\bark)$ is $L$-defined, so the minimal field of definition of $\Jac(F_\bark)$ is contained in the maximal purely inseparable subextension of $L$, which is $k'$.
Now any proper subfield of $k'$ of degree $3$ over $k$ can be written as $k(a)$ with $a^3\in k\setminus k^3$.
Then since $E/k$ is separable of degree $2$ and $k(a)/k$ is purely inseparable, we have $E\otimes_k k(a) \cong E(a)$.
This allows us to write
$$
F\otimes_k k(a) \cong F\otimes_E (E\otimes_k k(a)) \cong F\otimes_E E(a).
$$
Since $F$ and $E(a)$ are primitive purely inseparable extensions of $E$ of degree $3$, either $E(a) = F$ or $E(a)$ is linearly disjoint from $F$ over $E$. 
In either case we see that $F\otimes_k k(a)$ is a local ring, and hence $\Jac(F_\bark)$ cannot be $k(a)$-defined (since $F_\bark$ is not local).
Hence $k'$ is the minimal field of definition for $\Jac(F_\bark)$, as claimed. 
Finally, note that we also showed in this case that $F_{k''}$ is rigid for all intermediate fields $k\subset k'' \subset k'$.
\end{example}

\section{Application of high weight theory}\label{partii}
Let $k$ be a field and $G$ a smooth connected affine $k$-group. The $k$-unipotent radical $\RR_{u,k}(G)$ is the maximal smooth connected normal unipotent $k$-subgroup of $G$, and $G$ is \emph{pseudo-reductive} if
$\RR_{u,k}(G) = 1$. In this part we look to deploy \cite{BS22} and \cite{CPClass} to give a description of $k_V$ in \cref{thm:mainone}. 
Whilst we occasionally (as we have just done) emphasise the field $k$ by writing ``$k$-subgroup of $G$'', 
our convention throughout is that ``subgroup of $G$'' means ``$k$-subgroup of $G$''. 

\subsection{Preliminaries on the representation theory of pseudo-reductive groups}
We recall that if $k$ is a field and $\U$ is a unipotent $k$-group then we have \cite[Exp. XVII, Prop.~3.2]{SGA3}:
\begin{prop}\label{unipsimp}Let $\U$ be any unipotent $k$-group. Then the only simple $\U$-module is the $1$-dimensional trivial module, $k$.\end{prop}

This implies that if $k$ is a field and $V$ is a simple $G$-module then any normal unipotent subgroup of $G$ acts trivially on $V$. In particular, when $G$ is smooth and connected, the $k$-unipotent radical $\RR_{u,k}(G)$ of $G$ acts trivially on $V$ and in studying $V$ it does no harm to replace $G$ with its maximal \emph{pseudo-reductive} quotient $G/\RR_{u,k}(G)$ and so in what follows:

\begin{center}\framebox{$G$ will always denote a pseudo-reductive algebraic $k$-group over a field $k$.}\end{center}

The next sections recap the main results of \cite{BS22} which describe simple modules for pseudo-reductive groups by means of a high weight theory. 

\subsubsection{Induction} 
In \cite{BS22} simple modules for algebraic $k$-groups are constructed by induction.
We refer the reader to \cite[I.3.3]{Jan03} for the definition, and here just record the key property of \emph{Frobenius reciprocity}  \cite[I.3.4(b)]{Jan03}
for later use. 
Suppose that $H$ is a subgroup of $G$. Then for a $G$-module $V$ and an $H$-module $U$
\begin{equation}\label{eq:frob}
\Hom_G(V,\Ind_H^G(U)) \cong \Hom_H(\Res^G_H(V),U),
\end{equation}
where $\Ind_H^G(U)$ is the induced module and $\Res^G_H(V)$ is the $H$-module obtained from $V$ by restriction.

\subsubsection{Pseudo-reductive groups and Levi subgroups}\label{sec:pseudo-prelims} 
 There is a smallest extension $k\subseteq k'\subseteq \bark$ for which $\RR_{u,k'}(G_{k'})_{\bark}=\RR_u(G_{\bark})$, and $k'$ is called the \emph{minimal field of definition} of $\RR_u(G_{\bark})$, where $\RR_u(G_{\bar{k}})$ is the (geometric) unipotent radical of $G$; see \cite[Def. 1.1.6]{CGP15}. By \cite[Prop.~1.1.9]{CGP15} we have that 
the extension $ k'/k$ is finite and purely inseparable.

A pseudo-reductive group $G$ is called \emph{pseudo-split} if it contains a split maximal torus $T$;
in this case, there is a \emph{Levi subgroup} $M$ of $G$ containing $T$.
That is, there is a split reductive subgroup $M$ of $G$ containing $T$ and such that $G_{\bar{k}} = M_{\bar{k}} \ltimes \RR_u(G_{\bar{k}})$; see \cite[Thm. 3.4.6]{CGP15} or \cite[Thm. 5.4.4]{CPSurv}.

\subsubsection{Weil restriction}\label{sec:Weil} 
A key source of pseudo-reductive groups are Weil restrictions $\R_{k'/k}(G')$, where $k'/k$ is an inseparable field extension and $G'$ is a reductive $k'$-group.
We refer the reader to \cite[\S A.5]{CGP15} for an introduction to the key properties of Weil restriction tailored towards the study of pseudo-reductive groups.
See also \cite{Bosch_1990} for more details.

\subsubsection{Simple modules for pseudo-split groups}\label{sec:BS}
We summarise the main results from \cite{BS22}.
Suppose that $G$ is a pseudo-split pseudo-reductive $k$-group, with maximal torus $T$ and Cartan subgroup $C = Z_G(T)$. 
There is a Levi $k$-subgroup $M$ of $G$ containing $T$, and having chosen a Borel subgroup of $M$ (which for technical reasons should correspond to the negative roots), we get a system of positive roots, from which we can define a set of dominant weights $X(T)_+$ for $T$.
For each $\lambda\in X(T)_+$ there is a corresponding simple $M$-module $L_M(\lambda)$.
The following is a portmanteau theorem from \cite[Thm.~1.2, Thm.~3.1]{BS22}, with part (iii) following from the discussion in \cite[Sec.~4]{BS22}.
Note that the simple modules $L_G(\lambda)$ below are necessarily finite-dimensional, by local finiteness, but the induced modules $Q_G(\lambda)$ will not be finite-dimensional in general.

\begin{theorem}\label{thm:fromBS}
Suppose that $\lambda\in X(T)_+$. Let $Q_G(\lambda):= \Ind_M^G(L_M(\lambda))$.
\begin{itemize}
\item[(i)] The socle of $Q_G(\lambda)$ is a simple $G$-module, denoted by $L_G(\lambda)$. 
\item[(ii)] The assignment $\lambda \to L_G(\lambda)$ gives a one-one correspondence between the dominant weights $X(T)_+$ and the simple $G$-modules.
\item[(iii)] The highest weight space $L_G(\lambda)_\lambda$ of $L_G(\lambda)$ is a $C$-module isomorphic to $L_C(\lambda)$.
\item[(iv)] The restriction $\Res^G_M(L_G(\lambda))$ is isotypic and semisimple, and hence 
$$
\dim(L_G(\lambda)) = \dim L_C(\lambda)\cdot \dim L_M(\lambda).\label{thm:fromBS:dim}
$$
\end{itemize}
\end{theorem}

\begin{example}\label{rem:notabs}
It is noted in \cite[Rem 4.7(i)]{BS22} that the modules $L_G(\lambda)$ are rarely absolutely semisimple.
The most basic example is as follows: let $k$ be an imperfect field of characteristic $p$ and $G = \R_{k'/k}(\Gm)$ be the Weil restriction of the multiplicative group
along a purely inseparable extension $k'/k$ of degree $p$. (Then $G$ is pseudo-split and pseudo-reductive---see \cite[1.1.3]{CGP15}.) 
For $(r,p)=1$, the simple module $L_G(r)$ can be realised as the Weil restriction $\R_{k'/k}(V)$, where $V\cong k'$ is a $1$-dimensional vector space on which $\Gm$ acts with weight $r$; 
in other words, if $g\in \Gm(A)=A^\times$, then $g\cdot v\mapsto g^rv$ for any $v\in V(A)\cong A$.  
To see that this module is indeed simple, note that $G(k)=\Gm(k')=(k')^\times$ and so $G(k)$ has one orbit on the non-zero elements of $L_G(r)$. 
However, $L_G(r)_{k'}$ is reducible and indecomposable as a module for $G_{k'}$. 
(To see this latter statement, if we write $k' = k(a)$ with $k$-basis $1,a,\ldots,a^{p-1}$, then one can realise the elements of $G$ in their action on $L_G(r)$ explicitly as 
$p\times p$ matrices. Over $k'$ these matrices are trigonalisable.)

When $r=1$, this describes (the Weil restriction of) the natural action of $\Gm$ on $\Ga$ coming from scalar multiplication in $k'$.
See \cite[Sec.~5.2]{BS22} for further discussion, and \cref{cor:sepclosedsimple} below for a contrasting result for separable base changes of simple modules.
\end{example}

\subsubsection{Simple modules and the map \texorpdfstring{$i_G$}{iG}}\label{sec:i_G}
Keep notation from the previous section, so $G$ is a pseudo-split pseudo-reductive group with maximal torus $T$, Cartan $C = Z_G(T)$, and Levi subgroup $M$ containing $T$.
Let $K$ be the minimal field of definition of the unipotent radical of $G$, set $U = \RR_{u,K}(G_{K})$, and let $\G:= \R_{K/k}(M_{K})$.
The quotient $\pi:G_{K} \to M_{K}$ induces a map $i_G:G \to \G$, see \cite[Eq. (1.3.1)]{CGP15}.
This map plays a crucial role in the structure theory of pseudo-reductive groups developed in \cite{CGP15}.
We note that $i_G$ often has trivial kernel, but not always: see \cite[Ex. 1.6.3, Ex. 5.3.7]{CGP15} for examples of this.

It is proved in \cite[Prop 7.1.3(ii)]{CPSurv} that $\ker i_G$ is unipotent (with no non-trivial smooth and connected subgroups) and the smooth connected image $i_G(G)$ is pseudo-reductive. When the intersection of $\ker i_G$ with a Cartan subgroup of $G$ is trivial, we say that $G$ is of \emph{minimal type}.

It is also shown in \cite[Thm. 1.6.2(2)]{CGP15} that when $H = \R_{K/k}(H')$ is the Weil restriction of a reductive $K$-group $H'$ along the purely inseparable extension $K/k$, then the map $i_H$ is an isomorphism.
From this, it is not hard to deduce that  the map $i_{i_G(G)}$ is nothing other than the inclusion $i_G(G)\hookrightarrow \G$. 

Together with \cref{unipsimp}, this implies:
\begin{lemma}\label{lem:i_G}
  The action of $G$ on $L_G(\lambda)$ factors through $G\to i_G(G)$. 
  Thus we may assume $\ker(i_G)=1$ and in particular $G$ is of minimal type.
\end{lemma}

We make use of this observation in \cref{prop:splitend} below, comparing the simple modules of $G$ and $\G$.

\subsection{Endomorphisms of simple modules: pseudo-split case}\label{sec:endsimple1} 
In this and the following sections, we show how to pin down the endomorphism algebra $D = \End_G(V)$ more precisely, which allows us to be quite explicit about the field $k_V$ in \cref{thm:mainone}. In this section $G$ is pseudo-split with  a split maximal torus $T$ that is contained in a Levi subgroup $M$ and $C=Z_G(T)$ is a Cartan subgroup. We let $K$ denote the minimal field of definition of the unipotent radical of $G$ and thence get the canonical map $i_G:G\to\G:=\R_{K/k}(M_K)$, which restricts to $i_C:C\to\C:=\R_{K/k}(T)$. Evidently $T$ is a Levi subgroup of $C$ and $\C$; as $M$ is of $G$ and $\G$. 

It is helpful to relate the simple modules of $G$ and $\G$, building on work in \cite{BS22}.  Recall the following from \cite[Sec.~1]{BS22}.
\begin{defn}\label{def:Klambda}
Given any $\alpha\in \Z$ we can form the subfield $K(\alpha)$ of $K$ generated by $k$ together with $(K)^{\alpha}$. 
As $K/k$ is purely inseparable, taking $\alpha = p^e\beta$ with $\beta$ coprime to $p$, we have $K(\alpha) = K(p^e)$.
Now given a $T$-weight $\lambda = (\lambda_1,\ldots,\lambda_r)\in\Z^r$, let $K(\lambda)$ denote the purely inseparable 
subfield of $K$ generated by $k$ and the $K(\lambda_i)$.
\end{defn}

As explained in \cite[Thm.~5.8]{BS22} and its proof, the simple module $L_{\C}(\lambda)$ can be identified with the field $K(\lambda)$, in such a way that the action of $\C$ on $L_\C(\lambda)$ factors through the natural action of $\R_{K(\lambda)/k}(\Gm)$ on $\R_{K(\lambda)/k}(\Ga)$ via a surjection $\C\to \R_{K(\lambda)/k}(\Gm)$.

\begin{prop}\label{prop:splitend}
Suppose that $\ker i_G=1$ and identify $G=i_G(G)$ as a subgroup of $\G$. For any $T$-weight $\lambda$ we have:
\begin{enumerate}
\item The simple module $L_C(\lambda)$ can be identified uniquely with an intermediate extension $k\subseteq K_C(\lambda) \subseteq K(\lambda)$, 
and the action of $C$ on $L_C(\lambda)$ factors through the natural action of $\R_{K_C(\lambda)/k}(\Gm)$ on $\R_{K_C(\lambda)/k}(\Ga)$ by scalar multiplication.
\item More generally,  we have an isomorphism of $G$-modules $L_G(\lambda)\cong \R_{K_C(\lambda)/k}(L_M(\lambda)_{K_C(\lambda)})$, with $G$
acting through the Weil restriction $\R_{K_C(\lambda)/k}(\GL(L_M(\lambda)_{K_C(\lambda)}))$.
Further, by restriction, we have
$$
\Res^{\G}_G(L_{\G}(\lambda)) \cong L_G(\lambda)\oplus\cdots\oplus L_G(\lambda),
$$
with $[K(\lambda):K_C(\lambda)]$ summands on the right-hand side.\vspace{5pt}
\item \hspace{0.3\textwidth}$\End_G(L_G(\lambda))\cong K_C(\lambda).$\label{prop:splitend:end}
\end{enumerate}
\end{prop}

\begin{proof}
Since $T$ acts on $L_\C(\lambda) \cong K(\lambda)$ with the single weight $\lambda$, the restriction to $C$ has composition factors all equal to the simple module $L_C(\lambda)$.
In particular, there is a $C$-submodule $K_C(\lambda)$ of $K(\lambda)$ isomorphic to $L_C(\lambda)$.
By the argument in the proof of \cite[Lem.~5.10]{BS22}, and the simplicity of $K_C(\lambda)$, we see that 
we may assume that $K_C(\lambda)$ is the subfield of $K(\lambda)$ generated as a $C$-module by $1\in K_\C(\lambda)$.
Furthermore, \emph{loc.~cit.} shows that $C$ acts on it through the Weil restriction $\R_{K_C(\lambda)/k}(\Gm)$.
This proves (i).

When $G$ is commutative, so $G=C$ and $M=T$, then (ii) is an easy application of (i) plus Clifford's theorem \cite[I.6.16]{Jan03}: since $K(\lambda)$ is simple as a $\C$-module and $C$ is normal in $\C$, the restriction of $K(\lambda)$ to $C$ is semisimple, and it must be a direct sum of copies of $L_C(\lambda)$ for weight reasons.
Then (iii) follows since the $C$-action commutes with the the multiplicative structure of $K_C(\lambda)$ and any endomorphism $\phi:K_C(\lambda)\to K_C(\lambda)$ is completely determined by the image of $1\in K_C(\lambda)$.

With this in hand, we may prove (ii) and (iii) for general $G$. 
First, there is an obvious action of $G$ on the module $\R_{K/k}(L_M(\lambda)_K)$ through $\R_{K/k}(\GL(L_M(\lambda)_K))$,
since $\G$ acts on this module. 
Since the restriction of this module to $M$ is semisimple, and we have the correct highest weight, 
$L_G(\lambda)$ must identify with a submodule of this module.
Now the identification of the high weight space $L_G(\lambda_\lambda$ with the subfield $K_C(\lambda)\subseteq K$ shows
that we may assume that the high weight space of this copy of $L_G(\lambda)$ sits inside the natural copy of 
$\R_{K(\lambda)/k}(L_M(\lambda)_{K(\lambda)}) \subseteq \R_{K/k}(L_M(\lambda)_K)$.
It follows from \cref{thm:fromBS}(iv) that the vectors in the high weight space $L_G(\lambda)_\lambda$ 
generate $L_G(\lambda)$ under the action of a Levi subgroup $M\subseteq G$ containing $T$, for dimension reasons.
Since $\R_{K(\lambda)/k}(L_M(\lambda)_{K(\lambda)})$ is $M$-stable and has the correct dimension, it
must be isomorphic to $L_G(\lambda)$, and the first parts of (ii) follow.

Now we prove (iii). 
Any $G$-module homomorphism $\phi:L_G(\lambda)\to L_G(\lambda)$ must stabilise the high weight space $L_G(\lambda)_\lambda$ and is completely determined by what happens to it, again since its vectors generate the whole of $L_G(\lambda)$.
Hence $\phi$ is determined by its restriction to $L_G(\lambda)_\lambda\cong L_C(\lambda)$, which by the above shows that we have an inclusion $\End_G(L_G(\lambda))\subseteq K_C(\lambda)$.

On the other hand, any $G$-homomorphism $L_G(\lambda) \to Q_G(\lambda)$ must land in the simple socle $L_G(\lambda)$, so we have $\End_G(L_G(\lambda))\cong \Hom_G(L_G(\lambda),Q_G(\lambda))$.
But now we can use Frobenius reciprocity \eqref{eq:frob}:
$$
\Hom_G(L_G(\lambda),Q_G(\lambda)) \cong \Hom_M(\Res^G_M(L_G(\lambda)),L_M(\lambda)).
$$
Since the restriction $\Res^G_M(L_G(\lambda))$ is isomorphic to a direct sum of copies of $L_M(\lambda)$, we can conclude that $\End_G(L_G(\lambda)) \cong k^t$, where $t$ is the number of summands (recall that $\End_M(L_M(\lambda)) \cong k$ by \cite[Prop.II.2.8]{Jan03}). 
But the dimension formula in \cref{thm:fromBS}(iv) says that 
$t = \dim_k L_C(\lambda) = \dim_k K_C(\lambda)$, so we are done.

Finally to complete the proof for (ii), let $M\subseteq G\subseteq \G$ be a Levi subgroup. Then we have
\begin{align*}\Res^\G_M(L_\G(\lambda))&\cong L_M(\lambda)^{\dim L_\C(\lambda)}&{\text{(Thm.~\ref{thm:fromBS}(iv))}}\\
&\cong L_M(\lambda)^{[K(\lambda):k]}.
\end{align*}

Since also $\Res^G_M(L_G(\mu))\cong L_M(\mu)^{[K_C(\mu):k]}$ by \cref{thm:fromBS}(iv) applied to $G$, it follows that all the $G$-composition factors of $\Res^\G_G(L_\G(\lambda))$ are isomorphic to $L_G(\lambda)$ and there are $$[K(\lambda):k]/[K_C(\lambda):k]=[K(\lambda):K_C(\lambda)]$$ of them. It remains to show that $\Res^\G_G(L_\G(\lambda))$ is semisimple.

For that we calculate the dimension of $\Hom_G(\Res^\G_G(L_\G(\lambda)),L_G(\lambda))$, which is equal to that of $\Hom_G(L_\G(\lambda),\Ind_G^\G(L_G(\lambda)))$ by Frobenius reciprocity \eqref{eq:frob}. Now $$Q_\G(\lambda)=\Ind_M^\G(L_M(\lambda))\cong \Ind_G^\G(\Ind_M^G(L_M(\lambda)))=\Ind_G^\G(Q_G(\lambda))$$ by transitivity of induction. 
As $Q_\G(\lambda)$ has a simple socle, we get
\[\End_\G(L_\G(\lambda))\cong \Hom_\G(L_\G(\lambda),Q_\G(\lambda))\cong \Hom_G(\Res^\G_G(L_\G(\lambda)),L_G(\lambda))\]
by Frobenius reciprocity again. 
Part (iii) of this lemma identifies the left-most term as having the structure of $K(\lambda)$ as a $k$-vector space. 
Meanwhile, the right-most term has $k$-dimension equal to $\dim\End_G(L_G(\lambda))=[K_C(\lambda):k]$ times the multiplicity of $L_G(\lambda)$ in the head of $\Res^\G_G(L_\G(\lambda))$
(recall the \emph{head} of a module is the quotient by the radical). Therefore that multiplicity must be $[K(\lambda):K_C(\lambda)]$ and we are done.
\end{proof}

\begin{remarks}
(i). It follows from the proof of the first part of (ii) above that $L_G(\lambda)$ is \emph{canonically} embedded in $L_\G(\lambda)$, since $L_C(\lambda)$ is canonically embedded in $L_\C(\lambda)$ due to the inclusion $K_C(\lambda)\subseteq K(\lambda)$. 
  
 (ii). It is perhaps worth noting that part (ii) of the above result in no way requires $\RR_u(G_\bark)$ to be defined over $K_C(\lambda)$, even though the image of $G$ in the representation might have smaller (or even trivial) geometric unipotent radical. 
 A very simple example of this phenomenon is given as follows: let $G = \R_{k'/k}(\Gm)$, where $\Char(k) = 2$ and $k'/k$ is purely inseparable of degree $2$.
 Then the module $L_G(2)\cong k$ of high weight $2$ is $1$-dimensional, and $G$ acts on it through the squaring map, which in this case is a homomorphism $G\to \Gm$. 
 The kernel of this homomorphism is the non-smooth subgroup $\R_{k'/k}(\mu_2)$, whose base change to $\bark$ does contain $\RR_u(G_\bark)$.
 \end{remarks}

By \cref{prop:splitend}, $D:=\End_G(L_G(\lambda))\cong K_C(\lambda)$ is also a finite purely inseparable extension of $k$, and so by \Cref{lem:mfod}(i) it is 
itself the minimal field of definition of $\Jac(D_\bark)$.
By applying \cref{thm:Drigid}, we therefore obtain:
\begin{cor}\label{cor:splitrigid}
  With terminology as above, $k_V:=K_C(\lambda)$ satisfies the conclusion of \cref{thm:mainone}.
\end{cor}

\begin{example}\label{ex:klambda}
Easy examples show that the fields $K_C(\lambda)$ and $K(\lambda)$ in \cref{prop:splitend} can be different.
Let $k = F(a,b)$ be the field of rational functions in two indeterminates over a field $F$ of characteristic $p$,
and let $K = k(a^{1/p},b^{1/p})$, a purely inseparable extension of $k$ of degree $p^2$.
Set $s := a^{1/p}$ and $t:=b^{1/p}$.
Let $C = \R_{k(s)/k}(\Gm)\times \R_{k(t)/k}(\Gm)$, a commutative pseudo-split pseudo-reductive group with maximal split torus $T=\Gm\times\Gm$.
Then $K$ is the minimal field of definition of $\RR_u(C_\bark)$, so the group $\C = \R_{K/k}(T_K) \cong \R_{K/k}(\Gm)\times \R_{K/k}(\Gm)$,
with the factors of $C$ sitting naturally inside the factors of $\C$.

Consider the three $T$-weights $(1,0)$, $(0,1)$ and $(1,1)$.
The corresponding modules for $\C$ are all isomorphic as $k$-vector spaces to the field $K$ itself: since there are no nontrivial powers of $p$ appearing, we have $K(\lambda) = K$ for each of the three choices of $\lambda$.
On the other hand, we have $K_C(1,0) = k(s)$, $K_C(0,1) = k(t)$ and $K_C(1,1) = k(s,t) = K$,
and so we see that the field $K_C(\lambda)$ does depend on the weight and on the group $C$.
\end{example}

\begin{remark}\label{rem:finally}
Recall Example \ref{ex:finally}, which gives an example of a field $k$ and purely inseparable extensions $K/k$ and $E/k$ such that $K\otimes_k E$ is not rigid.
By setting $G = \R_{K/k}(\Gm)$ and $V = L_G(1)$, we obtain an example of a pseudo-split pseudo-reductive group $G$, a simple $G$-module $V$ with $\End_G(V) = K$, 
and an extension $E/k$ such that $V_E$ is not rigid.
This shows that we cannot hope that simple $G$-modules are absolutely rigid, even when $G$ is pseudo-split. 
Similar examples can be constructed replacing $\Gm$ with other split reductive groups, 
and will occur whenever we have purely inseparable extensions whose tensor product is not rigid.
\end{remark}

It is a fact---\!\!\cite[II.2.9]{Jan03}---that simple $G$-modules for split reductive $G$ are all defined over the relevant prime fields. With knowledge of the endomorphism ring in hand, we can give the generalisation to pseudo-split pseudo-reductive groups. 
At the same time we observe that the simple $G$-modules are absolutely indecomposable.

\begin{cor}\label{cor:sepclosedsimple}
Let $G$ be pseudo-split and $\lambda\in X(T)_+$ with $K_C(\lambda)$ the field associated with $L_G(\lambda)_\lambda\cong L_C(\lambda)$. 
\begin{itemize}
\item[(i)] We have that $L_G(\lambda)_E$ is isotypic with simple socle and head---hence indecomposable---for every field extension $E/k$.
\item[(ii)] We have that $L_G(\lambda)_E = L_{G_{E}}(\lambda)$ for any field extension $E/k$ which is linearly disjoint from $K_C(\lambda)$---for example, for any separable extension $E/k$.
\end{itemize}
\end{cor}

\begin{proof} 
By the results of \Cref{sec:morita}, the submodule structure of $L_G(\lambda)_E$ is controlled by the ideal structure of $K_C(\lambda)\otimes_k E$.
Since $K_C(\lambda)/k$ is purely inseparable, this tensor product is a Gorenstein local ring which implies the statement
(see \cref{sec:tensor}).
If in addition $K_C(\lambda)$ and $E$ are linearly disjoint, then $K_C(\lambda) \otimes_k E$ is a field and it
follows that $L_G(\lambda)_{E}$ is simple.
\end{proof}

\subsection{Endomorphisms of simple modules: general case}\label{sec:endsimple2}
We drop the assumption that $G$ is pseudo-split. Then for $V$ a simple $G$-module, $D:=\End_G(V)$ will no longer be a field in general, let alone a purely inseparable extension of $k$. We  elucidate the structure of $D$ taking inspiration from \cite{Tit71}.
A classification of the possible isomorphism classes of $D$ that could occur would subsume many difficult open questions about the Brauer groups of fields and we do not try to tackle this here. Instead, we assume we know the action of the absolute Galois group on (the Dynkin diagram of) $G$ and $V$. Then we are able to describe $D$ through its base change to a suitable separable extension---after which $G$ becomes pseudo-split and we can deploy the highest weight theory discussed in the last section. 
In particular we can calculate the dimension of $D$ based on this data.

Let $T$ be a maximal torus of $G$ and let $E/k$ be a finite Galois extension such that $S:=T_E$ is split, and hence $G_E$ is pseudo-split.
By choosing a Borel subgroup of a Levi subgroup of $G_E$ containing $S$, we can fix a system of dominant weights $X(S)_+$ in the weight lattice for $S$.
Let $\Gamma = {\rm Gal}(E/k)$, and let $W$ be the Weyl group of $G_E$.

We recall two actions of $\Gamma$ on the weights of $S$, see also \cite[Sec. 3.1]{Tit71}. 
The first arises from base change: given $\gamma\in \Gamma$ we can form the base change along $\gamma$ of the torus $S$,
giving a torus ${^\gamma}S$ defined functorially
by the formula ${^\gamma}S(A) = S(A\otimes_\gamma E)$ for each $E$-algebra $A$.
Since $k$ is fixed by $\gamma$, the tori $S$ and ${^\gamma}S$ have the common $k$-descent $T$, and so they are naturally isomorphic---we can identify ${^\gamma}S$ with $S$.
Under this identification, when we base change a character $\lambda$ of $S$ along $\gamma$ we obtain a new character ${^\gamma}\lambda$ of $S$.
Typically this will not preserve the dominant weights, but note that for each $\gamma\in \Gamma$ there is a unique $w\in W$ such that $w({^\gamma}X(S)_+)\subseteq X(S)_+$, and for $\lambda\in X(S)$ we set
\begin{equation}\label{eq:dot}
\gamma\cdot\lambda:=w({^\gamma}\lambda).
\end{equation}
Note that this is an action, and it respects the partial order on weights, since the system of positive roots corresponding to the choice of Borel subgroup above must also be preserved.

Let $V$ be a simple $G$-module.
Then $\Gamma$ acts semilinearly on $V_E = V\otimes_k E$ via its action on $E$;
denote this action by $v\mapsto \gamma(v)$. 
An $E$-subspace of $V_E$ has a $k$-form if and only if it is $\Gamma$-stable.
Note also that if $v\in V_E$ is a vector of $T_E$-weight $\lambda$,
then $\gamma(v)$ has weight ${^\gamma}\lambda$.
Since $V$ is simple as a $G$-module, $V_E$ is semisimple as a $G_E$-module by \cref{lem:sepext}, and so $V_E$ is
the sum of certain simple $G_E$-modules.
Let $\Lambda = \{\lambda=\lambda_1,\ldots,\lambda_r\}\subset X(S)_+$ be the set of highest weights occurring.
The following proof is based on that in \cite[Sec. 7.6]{Tit71}.

\begin{lemma}\label{lem:Vlambda}
Keep the notation above.
Further, for each $\lambda\in \Lambda$, let $V\{\lambda\}$ denote the sum of the simple submodules of $V_E$ isomorphic to $L_{G_E}(\lambda)$. Then:
\begin{itemize}
\item[(i)] $\Lambda$ forms a single $\Gamma$-orbit in $X(S)_+$;
\item[(ii)] $V_E$ is the direct sum of the $V\{\lambda\}$;
\item[(iii)] for each $\lambda_i\in\Lambda$, the multiplicity of $L_{G_E}(\lambda_i)$ in $V_{E}$ is a fixed integer, $d$.
\end{itemize}
\end{lemma}

\begin{proof}
Let $U$ be a simple $G_E$-submodule of $V$ isomorphic to $L_{G_E}(\lambda)$ and let $\gamma\in \Gamma$. The subspace $\gamma(U)$ is still a simple submodule, and any weight has the form ${^\gamma}\mu$ for a weight $\mu$ occurring in $U$.
Since the set of weights of $\gamma(U)$ is stable under the action of the Weyl group $W$, in fact any weight of $\gamma(U)$ has the form $\gamma\cdot\mu$ for
$\mu$ occurring in $U$. 
We have observed above that the ordering of weights is preserved by the map $\mu\mapsto \gamma\cdot\mu$,
so we can conclude that the module $\gamma(U)$ is isomorphic to $L_{G_E}(\gamma\cdot\lambda)$.

Now let $X$ be the (non-trivial) submodule of $V_E$ generated by the $\gamma(U)$ for $\gamma\in \Gamma$.
Then $X$ is $\Gamma$-stable, and hence has a $k$-form, which corresponds to a non-trivial $G$-submodule of $V$.
Since $V$ is simple, we conclude that $X = V_E$.
This proves (i), and (ii) follows because $V_E$ is semisimple.
For (iii), the above considerations imply that $\gamma(V\{\lambda\}) = V\{\gamma\cdot\lambda\}$ for each $\lambda\in \Lambda$ and $\gamma\in \Gamma$.
Hence the multiplicity of $L_{G_E}(\lambda)$ as a summand of $V\{\lambda\}$ must equal the multiplicity of $L_{G_E}(\gamma\cdot\lambda)$ as a summand of $V\{\gamma\cdot\lambda\}$.
\end{proof}

We can now describe the minimal field of definition of $\Jac(D_\bark)$ which serves the role of the field $k_V$ in \Cref{thm:mainone}.
According to the previous lemma and \Cref{cor:sepclosedsimple}(i), when we further extend to $k_s$ we get a decomposition 
\begin{equation}\label{ksdecomp}
V_{k_s} = \bigoplus_{\lambda\in \Lambda} L_{G_{k_s}}(\lambda)^d.
\end{equation}
For each $\lambda\in \Lambda$, denote by $K_{C,s}(\lambda)$ the purely inseparable extension of $k_s$ that identifies with the high weight space of $L_{G_{k_s}}(\lambda)$.
The compositum $K$ of the $K_{C,s}(\lambda)$ inside $\bar{k}$ is stable under the absolute Galois group, since the action of the Galois group permutes the summands, and hence permutes their high weight spaces.
This means that $K/k_s$ descends to a purely inseparable extension $k'/k$. 

\begin{theorem}
  With the above notation, $k'$ is the minimal field of definition of $\Rad_{G_\bark}(V_\bark)$ as a module. We may identify $k'$ with the minimal field of definition of $\Jac(D_\bark)$. It follows that we may take $k_V=k'$ in \cref{thm:mainone}.
  \end{theorem}
  \begin{proof}
    Let $\ks$ be the minimal field of definition of $\Jac(D_\bark)$. From the Morita equivalence in \cref{prop:REDE}: for every field extension $E/k$ we have that $D_E/\Jac(D_E)$ is absolutely semisimple if and only if $\Soc_{G_E}(V_{E})$ is absolutely semsimple; and since duality preserves the simple modules this happens if and only if $V_E/\Rad_{G_E}(V_E)$ is absolutely semisimple. Therefore $\ks$ identifies with the minimal field of definition of $\Rad_{G_\bark}(V_\bark)$.

Now \eqref{ksdecomp} implies
  \begin{equation}
    D_{k_s} \cong \bigoplus_{i=1}^r M_d(K_i).  \label{eq:ksdecomp}
  \end{equation}
  Since $K_i/k_s$ is purely inseparable, the minimal field of definition over $k_s$  of $\Jac(M_d(K_i)_\bark)$ is $K_i$ itself (\Cref{lem:mfod}(i)), and hence the minimal 
  field of definition of $\Jac(D_\bark)$ over $k_s$ is the compositum of these fields; that is, this minimal field of definition is $K\cong k'\otimes_k k_s$. Hence $k'\otimes_k k_s =\ks\otimes_k k_s$ as subfields of $\bark$ and so $k'=\ks$ as required. 
  
  For the last sentence, we apply \cref{thm:Drigid}.
  \end{proof}

Recall Definition \ref{defn:psplitting} of a $p$-splitting field for a $k$-algebra, and Remark \ref{rem:noncan}, which observes that there is no hope in general of such a field being unique.
In the light of this, the next theorem observes how restrictive the demand is for a division ring $D$ to be the endomorphism algebra $\End_G(V)$ for $V$ a simple $G$-module. In other words, the division algebras arising through \cref{lem:schur} come from a particularly special class.

\begin{theorem}\label{thm:torussplit}
  Let $D:=\End_G(V)$ for $V$ a simple $G$-module. Then there is a unique $p$-splitting field $E/k$ for $D$, and $E/k$ is Galois. This field identifies as the splitting field of a maximal torus $T$ in the image $\barG$ of $G$ in $\GL(V)$
\end{theorem}
\begin{proof} 
The existence of the unique splitting field $E/k$ for the torus $T$, and its property of being Galois is a consequence of the discussion in \cite[\S8.12]{Bor91}.\footnote{See \href{https://mathoverflow.net/questions/142801/splitting-field-of-a-torus}{MathOverflow 142801} for an in-depth analysis of this point.}
  
  Now $\barG_E$ is pseudo-split and pseudo-reductive, so the endomorphism algebra $D_E:=\End_{G_E}(V_E)$ is $p$-split by \cref{lem:Vlambda} and \cref{prop:splitend}\cref{prop:splitend:end}. For the converse, we assume $E$ is such that $D_E$ is $p$-split. Then for some $d$ and $r$ we have $D_E=\Mat_d(E_1)\times\dots\times \Mat_d(E_r)$ where the $E_i$ are purely inseparable over $E$. Thus $D_E$ has $r$ simple right modules, which remain simple after separable extension. By Morita equivalence, the same is true of the $\overline{G}_E$-module $V_E$: say $V_E\cong V_1^d\oplus \dots\oplus V_r^d$ with each $V_i$ being $k_s$-simple. Now $V_i$ identifies with some $L_{G_{k_s}}(\lambda)$ over $k_s$, and it descends to $E$, giving a high-weight module $L_{\overline{G}_E}(\lambda)$ on which $T_E$ acts completely reducibly. Iterating over the composition factors of $V_E$ we see that the image of $T_E$ in $\overline{G}_E$ is split, as required.  
\end{proof}

\begin{example}
  Let $k:=\F_p(t)$, $E:=\F_{p^2}(t)$, $k':=\F_p(t^{1/p})$, $F:=\F_{p^2}(t^{1/p})=E\otimes_k k'$. Then $\Gal(E/k)= \langle  \gamma \rangle =\Gal(F/k')$, say. Now for $n\geq 2$, denote by $\G$ the (reductive) $k$-group $\R_{E/k}(\SL_n)$ and take $G'$ to be the subgroup scheme of $\G$ given by $G'(A):=\{x\in \G(A)\mid x^\intercal\gamma(x)=1\}$ for any commutative $k$-algebra $A$. 

  The matrices $G'(k)$ describe the non-split reductive $k$-subgroup $\SU_n$;\footnote{In fact $G'$ is \emph{quasi-split} as it contains the descent of a Borel subgroup of $\GG$.} this means that the Galois group attached to $E/k$ induces a non-trivial involution of the Dynkin diagram of ${G'}_E$---see \cite[\S24.f]{Milne17}. Let $G:=\R_{k'/k}({G'}_{k'})$. We have that $G$ is pseudo-reductive and has a canonical $k$-subgroup $G'$ which is evidently a Levi subgroup for $G$. It is well known, and easy to calculate that the action of the non-trivial element $\gamma\in \Gal(E/k)\cong C_2$ on $X(T)_+$ is given by $\gamma: \lambda=(\lambda_1,\dots,\lambda_{n-1})\mapsto \gamma\cdot\lambda=(\lambda_{n-1},\dots,\lambda_1)$. 
  In particular if $\lambda=\varpi_1$ is the so-called first fundamental dominant weight---the one for which $L_{G_E}(\varpi_1)$ is the obvious $pn$-dimensional representation of the pseudo-split $G_{E}$---then there is a simple $G$-module $V=L_G\{\varpi_1,\gamma\cdot\varpi_1\}$ of dimension $2pn$ over $k$ such that $V_E\cong L_{G_E}(\varpi_1)\oplus L_{G_E}(\gamma\cdot\varpi_1)$. So $\End_{G_E}(V_E)\cong F\oplus F$ and is commutative. Since $F$ has an $E/k$-form $k'$, we see that the minimal field of definition of   $\Jac_{\bark}(D_\bark)$ is $k'$. Since the unique non-trivial field extension of $k$ contained in $k'$ is $k'$ itself, this shows that $Z(D)$---hence also $D$ and $V$---are all absolutely rigid.
\end{example}

\subsection{Applying the Conrad--Prasad classification}\label{lastsection}

We finish by applying the classification results of \cite{CPClass} to the problem of identifying $K_C(\lambda)$.
Since the formation of $K_C(\lambda)$ commutes with separable extension (Remark \ref{mfodsep}), we will assume for the whole section that $k=k_s$.

We describe in Section A.2 the so-called \emph{generalised standard construction} of a pseudo-reductive group, which applies to all groups $G$ locally of minimal type. 
In brief, this construction presents $G$ as a quotient 
\[
G\cong (\mathscr{G}\rtimes \Cs)/\C.
\]
Here $\mathscr{G} = \mathscr{D}(R_{k'/k}(G'))$ is the derived group of a $k'$-group $G'$, with $k'/k$ a non-zero finite reduced $k$-algebra and each fibre $G_i'$ of $G'$ a pseudo-reductive group from a prescribed list of types labelled (i), (ii) or (iii).
The other ingredients are  a commutative pseudo-reductive group $\Cs$ acting on $\mathscr{G}$ such that the Cartan subgroup $\C$ of $\mathscr{G}$ can be anti-diagonally embedded as a central subgroup of $\mathscr{G}\rtimes \Cs$; it is this central subgroup which is quotiented out, allowing for $\Cs$ to ``replace'' $\C$ and appear as the Cartan subgroup of $G$.

As observed in \cref{sec:i_G}, we may assume that $\ker i_G = 1$, so the Structure Theorem of A.2 applies and $G$ is presented as above.
Furthermore, the central quotient $\mathscr{G}\rtimes \Cs\to G$ does not change the root system, so $g$, $G'$ and $\mathscr{G}$ all have the same root system; this root system is reduced because $i_G$ embeds $G$ into a Weil restriction of a Levi subgroup and so $G'$ has no fibres of type (iii).
For fibres of types (i) and (ii), the Lemma in Appendix A.2 and the associated \cref{tabledata} describe the structure of a Cartan subgroup in terms of the root system  and certain linear algebraic data.
With these results in hand, we will specify $K_C(\lambda)$ just in terms of that data, together with the field $K_{\Cs}(\lambda_\Cs)$, where $\lambda_\Cs$ is the restriction of $\lambda$ to a maximal split torus $\Ts$ of $\Cs$---in other words, we have a complete understanding modulo the commutative case. 

Fix $G$ presented as above, and let us arrange that the isomorphism of a maximal torus $T'$ of $G'$ with some $(\Gm)^n\cong (\Gm)^{n_1}\times\dots\times (\Gm)^{n_r}$ is lined up with the descriptions in the table, so that $G_i'\cap T'$ is the factor $(\Gm)^{n_i}$ and its centraliser in $G_i'$ is the given Cartan subgroup.

Fix $i$ and let $\Delta=\Delta_>\sqcup\Delta_<$ be a base for the root system of $G_i'$. Let $e$, $e_>$ and $e_<$ denote the largest exponents of $p$ dividing $ ( \lambda,a ) $ for every $a\in \Delta$, $\Delta_>$ and $\Delta_<$ respectively. Then with reference to \cref{tabledata} we define

\begin{equation}K_i:=\begin{cases}
  k(k_i')^{p^{e}}\text{ if $G_i'$ is as in case (i)};\\
  k(K)^{p^{e}}\text{ if $G_i'$ is as in case (ii)(a)};\\
  k(K)^{p^{e_<}}\text{ if $G_i'$ is as in case (ii)(b);}\\
  \text{the compositum }k(K_>)^{p^{e_>}}(K)^{p^{e_<}}\text{ if $G_i'$ is as in cases (ii)(c) or (d)}.
\end{cases}\label{theki}\end{equation}

We come to the main theorem of this section.
\begin{theorem}\label{fieldfromclass}
  Let $\KK$ denote the compositum of all the $K_i$ together with $K_\Cs(\lambda_\Cs)$. Then 
\begin{equation}K_C(\lambda)= \mathcal{K}\label{theequation}\end{equation} 
\end{theorem}

Let $K/k$ be purely inseparable, $U$ a $k$-subspace of $K$ such that $U$ generates $K$ as a subfield, and put $\mathscr{U}:=U^*_{K/k}\subseteq \R_{K/k}(\Gm)=:\mathscr{K}$. The proof of the theorem above will require an understanding of the representations of groups $\mathscr{U}:=U^*_{K/k}\subseteq \R_{K/k}(\Gm)=:\mathscr{K}$. Such groups $\mathscr{U}$ are rather mysterious: for example, their dimensions seem to be impossible to predict easily---see \cite[9.1.8--9.1.10]{CGP15}. Nonetheless, the following lemma explains that their representation theory is easy. Being pseudo-reductive of rank $1$, the simple modules of $\mathscr{K}$ and $\mathscr{U}$ up to isomorphism are denoted $L_\mathscr{K}(\lambda)$ and $L_\mathscr{U}(\lambda)$ by \cref{thm:fromBS}; here, $\lambda$ indicates the weight of a maximal split torus. It follows that $\Res^{\mathscr{K}}_\mathscr{U}(L_\mathscr{K}(\lambda))$ is an isotypic direct sum of copies of $L_\mathscr{U}(\lambda)$. In fact:
\begin{lemma}\label{theustaractsirred}
   The restriction $\Res^{\mathscr{K}}_\mathscr{U}(L_\mathscr{K}(\lambda))$ is irreducible.
\end{lemma}
\begin{proof}
  If $\lambda=0$ then $L_\mathscr{K}(\lambda)\cong k$ is the trivial module and the result is clear. So suppose $\lambda\neq 0$. 
  
From \cite[Thm.~5.8]{BS22}, the action of $\mathscr{K}$ on $L_\mathscr{K}(\lambda)$ factors through $\mathscr{K}\to \mathscr{K}; x\mapsto x^a$ followed by the $p^s$-power map $\mathscr{K}\to \mathscr{K}^{p^s}\cong \R_{K^{p^s}/k}(\Gm)$ followed by an action of $\mathscr{K}^{p^s}$ on its natural module $L_{\mathscr{K}^{p^s}}(1)$---and the latter identifies with the field $kK^{p^s}$.

Let $\{u_1,\dots,u_d\}\in U\setminus\{0\}$ be a set of generators for $K$ as a $k$-algebra. Scaling as in \cite[3.1.4, Proof]{CPClass}, it does no harm to assume $u_1=1$, so that the ratios $u_i/u_j$---which are all $k$-points of $\mathscr{U}$---also contain a set of generators of $K$ as a $k$-algebra. The minimal $k$-subalgebra of $kK^{p^s}$ containing $1$ and stable under the group generated by the ratios $u_i^\lambda$ is the same as the subalgebra stable under $\mathscr{K}^{p^s}$; this is the whole of $K^{p^s}$ as required.
\end{proof}

\begin{proof}[Proof of \cref{fieldfromclass}]
Recall that we are working over $k=k_s$, and our assumption is that $G$ is of the form \cref{gisaquotient} and $V$ is a simple module for $G$. Set $D:=\End_G(V)$. We want to show that the minimal field of definition of $\Jac(D_\bark)$ is the compositum $\mathcal{K}$ of the fields referenced by the theorem. 
First observe that $V$ lifts to a simple module for the pseudo-split pseudo-reductive  semidirect product $\GG:=\G\rtimes\Cs$, through the quotient map $\G\rtimes\Cs\to G\cong(\G\rtimes\Cs)/\C$, by letting the central antidiagonal $\C$ in $\GG$ act trivially. We work with the Cartan subgroup $\CCC:=\C\times\Cs$, which is the centraliser of the maximal split torus $\TT:=T'\times \Ts$ where $T'$ is the canonical maximal split torus in $\R_{k'/k}(T')$. (Of course the product is direct since $\CCC$ is commutative.) This surjects onto $C$ with kernel $\C$, where $\TT$ maps onto $T$. 

Now for any $\lambda\in X(T)$ we get a corresponding lift $\widehat{\lambda}\in X(\TT)$, and so we get an isomorphism $V\cong L_{\G}(\widehat\lambda)$. Evidently $\End_G(V)=\End_{\G\rtimes\Cs}(V)$ and so these algebras equally identify with both $K_C(\lambda)$ and $K_\CCC(\lambda)$. Hence we need only show $K_\CCC(\lambda)\cong\KK$. Since $K_\CCC(\lambda)$ is by definition identified with the high weight space of $L_\GG(\widehat{\lambda})$ it suffices to show that this is the compositum $\KK$ as claimed. 

From \cref{prop:splitend} (for example) one sees that $V|_{G_i}$ is isotypic and semisimple; indeed it is a direct sum of copies of $L_G(\widehat{\lambda_i})$, where $\lambda_i:=\widehat{\lambda}|_{T_i}$. The  endomorphism algebra over $G_i$ is therefore the field $L_{C_i}(\lambda_i)$ and we wish to see that this identifies with the field $K_i$ in \cref{theki}, which we now do case-by-case. 

In case (i) $C_i=\R_{k_i'/k}(T_i')$ and the statement that $L_{C_i}(\lambda_i)\cong K_i$ is \cite[Thm.~5.8]{BS22}. The same result also deals with case (ii)(a).  Then we have $p=2$. We treat case (ii)(d), the others being similar. 

By \cref{theustaractsirred} the Cartan subgroup $(V_>)^*_{K_>/k}\times (V_<)^*_{K/k}$ of $G_i$ has the same irreducible representations as $\R_{K_>/k}(\Gm)\times \R_{K/k}(\Gm)$, whence we can appeal again to \cite[Thm.~5.8]{BS22}. Now apply \cref{directproduct} (inductively) to the product $\GG=\prod G_i\rtimes\Cs$ to see that $K_C(\lambda)\cong\End_\GG(V)$ is the compositum $\KK$ of the fields $K_i$ together with $K_\Cs(\lambda_\Cs)$ as required.
\end{proof}

The Conrad--Prasad structure theorem also gives a refinement to our dimension formula in \cref{thm:fromBS}(iv). 
As in the proof of \cref{fieldfromclass}, we may lift the action of a pseudo-split $G$ on $L_G(\lambda)$ to that of $\GG=\G\rtimes \Cs$, where we have accordingly a decomposition $\TT=T'\times \Ts$ of a split maximal torus of $\GG$. Let $M\supseteq \TT$ denote a split Levi subgroup of $\GG$. Then $M\cap \Cs=\Ts$ and let $M_i:=M\cap G_i$ with $T_i:=T\cap G_i$ a corresponding maximal split torus; lastly, set $\lambda_i$ (resp.~$\lambda_\Ts$) the restriction of $\lambda$ to $T_i$ (resp.~$\Ts$). Since the $M_i$ are absolutely simple and simply connected, we have $M\cong M_1\times\dots\times M_r\times \Ts$, and $\End_{M_i}(L_{M_i}(\lambda_i))=k$ is a trivial $M$-module. Using \cref{directproduct} and \cref{redtofactors} it follows that $L_M(\lambda)\cong L_{M_1}(\lambda_1)\otimes\dots\otimes L_{M_r}(\lambda_r)\otimes L_{\Ts}(\lambda_{\Ts})$. Since  $L_{\Ts}(\lambda_{\Ts})$ is just a $1$-dimensional weight module $k_{\lambda_\Ts}$ with $t\cdot x=\lambda(t)x$ for $t\in \Ts(k)$, it follows that $\dim L_M(\lambda)=\prod_{i=1}^r \ell_i$ where $\ell_i=\dim L_{M_i}(\lambda_i)$. (All of this is  well-known.)
  \begin{cor}\label{cor:dimforsimples}
 We have \begin{equation}L_\GG(\lambda)|_M\cong \left(L_{M_1}(\lambda_1)\otimes\dots\otimes L_{M_r}(\lambda_r)\otimes L_{\Ts}(\lambda_\Ts)\right)^{\oplus \dim \mathcal{K}}.\label{tensorproddec}\end{equation}
Hence $\dim L_\GG(\lambda)=\prod \ell_i\cdot [\mathcal{K}:k]$, where $\mathcal{K}$ is the compositum in \cref{fieldfromclass}.

If $G$ is not necessarily pseudo-split, we have $\dim V=\dim_{k_s} L_{G_{k_s}}(\lambda)\cdot d\cdot |\Lambda|$, where $\Lambda$ is the orbit of $\Gal(k_s/k)$ on $\lambda$, some composition factor $L_{G_{k_s}}(\lambda)$ of $V_{k_s}$ occurs with multiplicity $d>0$, and where $\dim_{k_s}L_{G_{k_s}}(\lambda)$ can be deduced from the previous formula.

Furthermore, $\End_\GG(L_\GG(\lambda))\cong \mathcal{K}$, so that $\dim\End_\GG(L_\GG(\lambda))=[\mathcal{K}:k]$. In the non-pseudo split case, with notation of \cref{eq:ksdecomp}, we have $\dim \End_G(V)= d^2\cdot |\Lambda|\cdot [K_1:k_s]$.
\end{cor} 
\begin{proof}
For the first statement, note the description of $L_M(\lambda)$ has already been established. Applying the formula from \cref{thm:fromBS}(iv) tells us that  $\dim L_\GG(\lambda)=\dim L_M(\lambda)\cdot \dim L_C(\lambda)$; but $L_C(\lambda)\cong K_C(\lambda)\cong \KK$ by \cref{fieldfromclass}.

The second statement is immediate.
\end{proof}

\section*{Appendix: a very brief guide to the Conrad--Prasad classification}

The classification of pseudo-reductive groups in \cite{CGP15}, yet slightly incomplete, is an intricate theorem. At the request of the referee we offer a \textit{de minimis} description from our own modest understanding. We wish for it to provide the scaffolding necessary to understand \cref{lastsection} 
of our paper; we hope it is useful in sparking the reader's interest sufficiently for them to mount an assault on the comprehensive account in \cite{CPSurv}---or even the source material as found in \cite{CPClass,CGP15}. 

\subsection*{A.1 Standard groups}
Let $G$ be a \emph{pseudo-reductive} group over $k$, that is to say it is a smooth connected affine $k$-group scheme with $k$-unipotent radical $\RR_{u,k}(G)=1$. A \emph{pseudo-simple} group is a smooth connected affine $k$-group defined by having no non-trivial smooth connected normal subgroup; this is of course a special type of pseudo-reductive group. Analogously with reductive groups, one wishes to proceed by reducing the pseudo-reductive case to the pseudo-simple and classifying the latter.

The headline result \cite[Thm.~1.5.1]{CGP15} is a first pass at this which works as least as long as $\Char k=p>3$; it is called the \textit{the standard construction} \cite[\S1.3--1.5]{CGP15}. Since the standard construction is a necessary concept also in the full classification as given in \cite{CPClass}, we record the platonic ideal first. For elegance among other things, one takes $k'$ to be a non-zero finite reduced $k$-algebra; so there is an isomorphism with a product of fields $k'\cong k_1'\times \dots \times k_r'$, with each $k_i'/k$ a finite field extension. 
Then a (reductive) $k'$-group $G'$ is isomorphic to a product $G'\cong G'_1\times\dots\times G'_r$, where each $G'_i$ is a (reductive) $k_i'$-group and $G'_i$ is the fibre of $G'$ over $k_i'$; we have corresponding direct product decompositions for the centre $Z':=Z_{G'}$, for a maximal torus $T'$, and for the Weil restriction $\R_{k'/k}(G')$.
Assuming that $G'$ is reductive, the Weil restriction $\R_{k'/k}(G')$ is always pseudo-reductive and has Cartan subgroup $\R_{k'/k}(T')$ \cite[Prop.~1.1.10, Prop.~A.5.15(3)]{CGP15}. Since $T'$ and $T'/Z'$ are both reductive, it follows that $\R_{k'/k}(T')$ and $\R_{k'/k}(T'/Z')$ are pseudo-reductive and they are obviously commutative. 
The standard construction swaps out the Cartan subgroup $\R_{k'/k}(T')$ of $\R_{k'/k}(G')$ for any commutative pseudo-reductive group $C$ that sits in a factorisation $\R_{k'/k}(T')\stackrel{\phi}{\to} C\to\R_{k'/k}(T'/Z')$. We start by making $\R_{k'/k}(G')$ bigger, noting that $C$ acts on $\R_{k'/k}(G')$ by conjugation so that we can form the semidirect product $\R_{k'/k}(G')\rtimes C$; then we tighten things back up again by killing the diagonal subgroup $(x^{-1},\phi(x))$. The group thus formed, namely 
\[\frac{\R_{k'/k}(G')\rtimes C}{\R_{k'/k}(T')},\] is \textit{standard} \cite[Defn.~1.4.4]{CGP15}. 

As advertised, most of the time (e.g. away from characteristics $2$ and $3$) this gives \emph{all} pseudo-reductive groups. 
The way this goes is to show that $G$ is standard if and only if its derived subgroup $\D(G)$ is \cite[Prop.~5.2.1]{CGP15}, and that $\D(G)$ is standard if its base change to $k_s$ is standard \cite[Cor.~5.2.3]{CGP15}, which reduces us to the case that $k=k_s$, $k'/k$ is a purely inseparable field extension and $G$ \emph{pseudo-split}---i.e. it has a split maximal torus $T$.
Now one can show that such a $G$ is standard if each of its normal non-commutative pseudo-simple factors are \cite[Prop.~5.3.1]{CGP15}, and so one wishes to classify the standard such groups.

A major step is to show that any pseudo-split pseudo-reductive group has a Levi subgroup $M$---a subgroup isomorphic over $\bar k$ to the largest reductive quotient of $G$, see \cite[Thm.~3.4.6]{CGP15} or a simpler proof in \cite[Thm.~5.4.4]{CPSurv}. 
As noted in the main body of the paper (Section \ref{sec:i_G}), if $K/k$ is the minimal field of definition of the geometric unipotent radical of $G$, then the quotient map $G_K\to M_K$ induces a map $i_G:G\to \R_{K/k}(M_K)$, where $\R_{K/k}$ is the Weil restriction functor. 
In all cases where the root system of $G_{k_s}$ is reduced (so away from characteristic $2$ or avoiding type $BC_n$ in characteristic $2$), the kernel of $i_G$ is central in $G$ \cite[Prop.~2.3.4]{CPClass}, which helps us to compare the root groups of $G$ with those of $\R_{K/k}(M_K)$.
It follows in particular that the root groups of $G$ identify as $k$-vector subspaces of $K$. 
If $t$ is an element of such a $k$-space, then $x_a(t)$ denotes a corresponding element of the root group. 
Although the root groups can have very high dimension, by conjugating with the Weyl group, one at least sees that all roots of the same length in each irreducible component of the root system must support isomorphic root groups; we can push this quite a bit further using the Steinberg relations in $M_K$ (or, in what amounts to the same thing, $\R_{K/k}(M_K)$). Recalling the elements $n_a(t)=x_a(t)x_{-a}(-t^{-1})x_a(t)$ and $h_a(t)=n_a(t)n_a(-1)$, we can see that a Cartan of $G$ must contain (the Zariski closure of) elements $h_a(t/u)$ for $t$ and $u$ non-zero elements of the root group $U_a$ (that we have identified with a $k$-subspace of $K$). Thence one can use the commutator formula and the relation $x_b(t)^{h_{a}(u)}=x_b(tu^{\langle b,a^\vee\rangle})$ to import the pieces of an $A_1$-subgroup corresponding to a root $a$ into a $b$-root group. Crucially, \emph{in the generic case that there is some $a$---possibly $b$ itself---such that $\langle b,a^\vee\rangle$ is an invertible integer in $k$}, we end up finding that the root groups are all fields. This condition holds unless $p=2$ and the root system is of type $B_n$ or $C_n$. If there are degeneracies in the commutator relations that mean that the short root groups of $M$ generate a subgroup with those as its roots, then the fields over short and long roots can be different and standardness does not hold. 

Conversely, one can conclude standardness from knowing $\Phi(G)=\Phi(M)$ (and hence that the former is reduced) and that the root groups are all the same field, which happens so long as $(\Phi(M),p)\neq (B_n,2), (C_n,2), (F_4,2), (G_2,3)$, ($n\geq 1$). 

\subsection*{A.2 Beyond standard groups}

While \cite{CGP15} goes a long way towards removing barriers to a full classification, \cite{CPClass} is able to push this further and identify essentially the only remaining barrier as: \emph{classify the non-minimal rank one groups in characteristic $2$.} 
Among the non-standard groups:
\begin{itemize}
\item[(a)] If we have a root system of type $F_4$ in characteristic $2$ or $G_2$ in characteristic $3$, we can construct groups where the
the long root field is any finite field extension of the short root field of exponent $1$.
The resulting pseudo-simple groups are called \emph{basic exotic groups} \cite[Defn.~7.2.6]{CGP15}, \cite[Def.~2.2.2]{CPClass}. The calculations  that fundamentally make this work are: that root groups corresponding to short roots $a$ and $b$ such that $a+b$ is long commute; and that for any long root $c$, $x_c(t)^{h_a(u)}=x_c(tu^r)$ where $r$ is a power of $p$.

However, two facts come to the rescue here to exclude anything else: one, that in these root systems, the roots of a single length give an $A_2$ or $D_4$ subsystem, hence correspond to standard subgroups themselves; two, that these groups are both adjoint and simply connected causes them to split off as direct factors. So suitably altering `standard' to an improved `generalised standard' we can consider these cases covered.

\item[(b)] Similar degeneracies in the commutator relations lead also to basic exotic groups of type $B_n$ and $C_n$ in characteristic $2$---including the $n=1$ case---but they cannot be rescued because the short (resp. long) roots give subsystems of type $A_1^n$. 
In type $B_n$ (resp.~$C_n$) in characteristic $2$, the long (resp.~short) root groups identify with a field $K\supseteq k$ over which the short (resp.~long) root groups need only identify with a $K$-vector subspace of an exponent-$1$ field extension $E$ of $K$. This leads to the definition of a class of \emph{generalised basic exotic} groups in types $B$ and $C$.
When $n=2$ there is yet another class of \emph{generalised basic exceptional} groups in type $B_2=C_2$ because of the collision of degeneracies for both root lengths. The concrete construction here rests on the analysis of degenerate quadratic forms in characteristic $2$, see \cite[Chap.~7, Chap.~8]{CPClass}.

\item[(c)] Lastly we may have that $G$ has a non-reduced root system of type $BC_n$ and $p=2$. If $\ker i_G=1$ then these can be classified; the reader can refer to \cite[Thm.~9.8.6]{CGP15} (or \cite{BRSS24} for a more elementary construction).
\end{itemize}

If we look at the rank $1$ characteristic $2$ case  there is as-yet-incomplete information. 
See \cite[Sec.~3.1, Sec.~4.2]{CPClass} for a zoo of examples, and detailed explanations of how they limit the possibilities for a complete classification result.
To get a clean result, the rank-1 case must be controlled with an extra hypothesis.

Let $H$ be pseudo-reductive and $C$ a Cartan subgroup of $G$, then $H$ has \emph{minimal type} if $(\ker i_H)\cap C$ is trivial (this holds in particular if $\ker i_H = 1$) \cite[Defn.~9.4.4]{CGP15}, \cite[Defn.~2.3.2]{CPClass}.
 Now a pseudo-reductive $G$ is \emph{locally of minimal type} if for each root $a$ the rank-$1$ subgroup $(G_{k_s})_{\pm a}$ of $G_{k_s}$ generated by the roots groups $(G_{k_s})_a$ and $(G_{k_s})_{-a}$ has some pseudo-simple central extension which is of minimal type, \cite[Defn~4.3.1]{CPClass}.
 The main result of the Conrad-Prasad classification is that there is a generalisation of the standard construction above which captures all such groups, \cite[Thm.~9.2.1]{CPClass}.

\begin{theorem*}[Structure Theorem]
  Let $G$ be a pseudo-reductive group over a field $k$. 
  Then $G$ is generalised standard if and only if it is locally of minimal type.
\end{theorem*}

We therefore need to describe the generalised standard construction. 
First suppose that $k'/k$ is a non-zero finite reduced $k$-algebra and let $G'$ be a $k'$-group. 
Then $(G',k'/k)$ is said to be a \emph{primitive pair} \cite[Defn.~9.1.5]{CPClass} if each fibre $G'_i$ is one of the following:
\begin{enumerate}
  \item a connected semisimple, absolutely simple, and simply connected $k_i'$-group;
  \item 
  \begin{enumerate}
  \item a basic exotic group of type $G_2$ ($p=3$) or $F_4$ ($p=2$);
  \item a generalised basic exotic group of type $B$ ($p=2$);
  \item a generalised basic exotic group of type $C$ ($p=2$);
  \item a rank-$2$ basic exceptional group of type $B_2$ ($p=2$);
  \end{enumerate}
  \item a minimal-type absolutely pseudo-simple group with a non-reduced root
  system and root field equal to $k_i'$ ($p=2$).\label{nonredcase}
\end{enumerate}
If $C$ denotes a Cartan subgroup of $G$ then the $k$-group functor $$\underline{\Aut}_{G,C}:A\mapsto \{f\in \Aut_A(G_A)\mid f|_{C_A}=\id_{C_A}\}$$ is affine of finite type and has maximal smooth closed $k$-subgroup $Z_{G,C}$ by \cite[2.4.1]{CGP15}. With this notation, we say $G$ is \emph{generalised standard} \cite[Defn.~9.1.7]{CPClass} if there is a $4$-tuple $(G',k'/k,T',\Cs)$ such that $(G',k'/k)$ is a primitive pair, $T'$ a maximal torus of $G'$, $\Cs$  a commutative pseudo-reductive group, and there is a factorisation 
\begin{equation}\mathscr{C}\stackrel{\phi}{\to} \Cs\stackrel{\psi}{\to}Z_{\mathscr{G},\C}=\R_{k'/k}(Z_{G',C'})\label{fact}\end{equation}
with $\mathscr{G}=\mathscr{D}(R_{k'/k}(G'))$, $C'=Z_{G'}(T')$,  and $\C=\mathscr{G}\cap \R_{k'/k}(C')$---a Cartan $k$-subgroup of $\mathscr{G}$---such that there is a $k$-isomorphism 
\begin{equation}(\mathscr{G}\rtimes \Cs)/\C\cong G\label{gisaquotient}\end{equation}
where $\C$ is anti-diagonally embedded as a central $k$-subgroup of $\mathscr{G}\rtimes \Cs$.

\subsection*{A.3 Cartan subgroups in minimal type groups}
 
For use in the main body of the paper, we wish to describe the Cartan subgroups arising from fibres of types (i) and (ii) following the Structure Theorem in the previous section.
This is simply collecting together information from \cite{CGP15, CPClass}.
The Cartans for type (iii) are described by \cite[(9.7.6), (9.8.2)]{CGP15}; see also the preamble to \cite[Thm.~4.3]{BRSS24}.

Let then $k'$ be a finite field extension of $k$. 
The cases for $G'$ in (ii) are pseudo-split pseudo-simple $k'$-groups that depend on some linear algebra data we describe. Let $K$ be the minimal field of definition of $G'$ and let $G''$ be a split simply connected $K$-group with the same root system. Then $G'$ is constructed inside $\R_{K/k'}(G'')$ by specifying the root groups of one as  subspaces of those of the other.
If $V$ is a $k$-subspace of $K$, then we let $(V)^*_{K/k}$ denote the 
Zariski closure in $\R_{K/k}(\Gm)$ of the ratios of non-zero elements of $V$. 
Finally, if $\Delta$ is a (reduced) root system, we let $\Delta_>$ denote the long roots and $\Delta_<$ the short roots.
We have:

\begin{lemma*}[Structure of Cartan]\label{lem:Cartan_structure}
  Suppose that $k=k_s$ and $(G',k'/k)$ is a primitive pair of type (i) or (ii). Let $T'$ be a split maximal torus of $G'$ and $K$ the minimal field of definition of $\RR_u((G')_\bark)$. Then \cref{tabledata} describes a Cartan subgroup of $\D(\R_{k'/k}(G'))$.
\end{lemma*}

\begin{table}\footnotesize \begin{tabular}{l|l|l|l}
Case & Input Data & Additional Facts & Cartan of $\D(\R_{k'/k}(G'))$ \\\hline\hline
(i) & root system & & $\R_{k'/k}(T')$\\\hline

(ii)(a) & \parbox[t]{0.22\textwidth}{root system; $K$} & 
\parbox[t]{0.25\textwidth}{$(G')_a\cong \R_{K/k'}(\Ga)$, $a$ long;\\ $(G')_a\cong \Ga$,  $a$ short} & $\R_{K/k}(T')$\\\hline

(ii)(b) & \parbox[t]{0.22\textwidth}{rank; $K$; \\ $k'$-subspace $V$ of $K$\\ such that $k' (  V )  = K$} & 
\parbox[t]{0.25\textwidth}{$(G')_b=\underline{V}$, $b$ short; \\ $(G')_a\cong \Ga$, $a$ long} & 
\parbox[t]{0.37\textwidth}{$\left(\prod_{a\in\Delta_>} a^\vee (\R_{k'/k}(\Gm))\right) \\\phantom{something}\times (\R_{K/k}(b_K^\vee))(V_{K/k}^*)$\\where $\Delta_<=\{b\}$}\\\hline

(ii)(c) & \parbox[t]{0.22\textwidth}{rank; $K$; \\ $k'$-subspace $V_>$ of $K$,\\defining subfield\\ $K_>=k' (  V_> ) $} & 
\parbox[t]{0.25\textwidth}{$(G')_b=\underline{V_>}$, $b$ long;\\ $(G')_a\cong \R_{K/k'}(\Ga)$, $a$ short} & \parbox[t]{0.37\textwidth}{$\left(\prod_{a\in\Delta_<} a^\vee (\R_{K/k}(\Gm))\right)$ \\ \phantom{something} $\times (\R_{K/k}(b_K^\vee))((V_>)_{K/k}^*)$ \\where $\Delta_>=\{b\}$}\\\hline

(ii)(d) & \parbox[t]{0.22\textwidth}{$K$; \\ $k'$-subspace $V_>$ of $K$,\\defining subfield\\ $K_>=k' (  V_> ) $;\\
$K_>$-subspace $V_<$ of $K$\\ with $K=k' (  V_< ) $} & 
\parbox[t]{0.25\textwidth}{$(G')_b=\underline{V_>}$, $b$ long;\\
$(G')_a=\underline{V_<}$, $a$ short} & 
\parbox[t]{0.32\textwidth}{$(V_>)_{K_>/k}^*\times (V_<)^*_{K/k}$}\\\hline
\end{tabular}\caption{Data describing $G'$ and Cartan subgroups in cases (i) and (ii)\label{tabledata}}\end{table}
\begin{proof}
The table shows, for each of the five cases: the information needed to construct $G'$ in column 2; 
further facts about how this data feeds into the roots groups in column 3;
the resulting Cartan subgroup in column 4.  

In type (i), the group $G'$ is semisimple, hence has Cartan subgroup $T'$; 
in type (ii)(a), the Cartan $k'$-subgroup $C'=Z_{G'}(T')$ of $G'$ is described by \cite[(8.1.1)]{CGP15};
in types (ii)(b) and (c) with rank at least $2$, see \cite[8.2.5]{CPClass};
in type (ii)(d), see \cite[8.3.7]{CPClass}.
The definitions \cite[8.1.1, 8.2.3]{CPClass} of generalized basic exotic groups of types $B$ and $C$ put the rank $1$ groups in the type $B$ case, where we end up considering groups of the form described by \cite[3.1.4]{CPClass}, where the Cartans are also described.
In all cases, the Cartan subgroups can be deduced from those given in \cref{tabledata} by taking $k=k'$. 

In general, $\R_{k'/k}(C')$ is a Cartan subgroup of $\R_{k'/k}(G')$, \cite[A.5.15(3)]{CGP15}. In cases (i) and (ii)(a), the latter group is perfect by \cite[1.3.4, 8.1.2]{CGP15} and the result is immediate; so we may assume we are in one of the remaining three types and $p=2$. 
If $\rk(G')=1$, then as mentioned above we are in the case described by \cite[3.1.4]{CPClass}. 
Then $\D(\R_{k'/k}(G'))$ has Cartan subgroup $V^*_{K/k}$. 
If $\rk(G')=n>2$, then the $A_{n-1}$ subgroup generated by the long (resp.~short) root groups of $G'$ in case (ii)(b) (resp.~(ii)(c)) is isomorphic to $\SL_{n}$ (resp.~$\R_{K/k'}(\SL_{n})$) and so its Weil restriction is perfect. Hence, the long root (resp.~short root) factor of a Cartan subgroup of $\R_{k'/k}(G')$ survives after passing to the derived subgroup. This reduces the assertion to rank $1$, which we have established already---one may appeal to \cite[C.2.32]{CGP15} if so desired. Similar arguments apply when $\rk(G')=2$ and are omitted.
\end{proof}

{\bf{Acknowledgement:}} Many thanks to Brian Conrad for spotting an error in the first version of this paper, as well as helpful clarifications about the main theorem of \cite{CPClass}. The second author is supported by the Leverhulme Trust Research Project Grant number RPG-2021-080. 

{\bf Conflict of interest statement:} On behalf of all authors, the corresponding author states that there is no conflict of interest.

{\bf Data availability statement:} There is no data associated with this manuscript.

{\footnotesize
\bibliographystyle{amsalpha}
\bibliography{bib}}

\end{document}